\documentclass{article}
\RequirePackage[numbers]{natbib}
\usepackage{soul}
\usepackage{amsmath, amssymb, amsfonts, amsthm}
\usepackage{mathrsfs}
\usepackage{enumitem}
\usepackage{stmaryrd}
\usepackage[latin1]{inputenc}
\usepackage[T1]{fontenc}
\usepackage{color}
\usepackage{graphicx}
\usepackage{bm}
\usepackage{hyperref}
\usepackage{wrapfig}
\usepackage{cleveref}

\theoremstyle{plain}
\newtheorem{theorem}{Theorem}[section]

\newtheorem{proposition}[theorem]{Proposition}
\newtheorem{lemma}[theorem]{Lemma}
\newtheorem{corollary}[theorem]{Corollary}
\theoremstyle{remark}
\newtheorem{remark}[theorem]{Remark}

\usepackage{mathrsfs, stmaryrd, mathtools}

\DeclarePairedDelimiterX\intff[2]{[}{]}{#1,#2}
\DeclarePairedDelimiterX\intfo[2]{[}{)}{#1,#2}
\DeclarePairedDelimiterX\intof[2]{(}{]}{#1,#2}
\DeclarePairedDelimiterX\intoo[2]{(}{)}{#1,#2}
\DeclarePairedDelimiter{\pars}{(}{)}
\DeclarePairedDelimiter{\bracks}{[}{]}
\DeclarePairedDelimiter{\absolute}{|}{|}
\DeclarePairedDelimiter{\braces}{\lbrace}{\rbrace}
\DeclarePairedDelimiterX{\setof}[2]{\lbrace}{\rbrace}{#1\,{:}\,#2}
\DeclarePairedDelimiterX{\bracksof}[2]{[}{]}{#1\,\delimsize\vert\,#2}
\DeclarePairedDelimiterX{\parsof}[2]{(}{)}{#1\,\delimsize\vert\,#2}
\DeclarePairedDelimiterXPP\lnorm[2]{}\lVert\rVert{_{#1}}{#2}

\newcommand{\gnP}{\mathbb P}
\newcommand{\gnE}{\mathbb E}
\newcommand{\green}[1]{G_{#1}}
\newcommand{\greenkill}[2]{G_{#1}^{(#2)}}
\newcommand{\xn}[1]{X_{#1}}
\newcommand{\rwS}[1]{S_{#1}}
\newcommand{\rwD}{\eta}
\newcommand{\rwP}[2]{\mathsf P^{#1}_{#2}}
\newcommand{\rwE}[2]{\mathsf E^{#1}_{#2}}
\newcommand{\rwVar}[1]{\Gamma_{#1}}
\newcommand{\rwJ}[1]{J_{#1}}
\newcommand{\capa}{\mathtt{cap}_{\rwD}}
\newcommand{\gnT}{T}
\newcommand{\gnNd}{u}
\newcommand{\gnCh}[1]{k_{#1}}
\newcommand{\gnD}[1]{\mathbf d_{#1}}
\newcommand{\gnX}[1]{X_{#1}}
\newcommand{\gnR}{R}
\newcommand{\brwCh}{\mu}
\newcommand{\brwD}{\theta}
\newcommand{\brwP}{P_{\brwCh,\brwD}}
\newcommand{\brwE}{E_{\brwCh,\brwD}}
\newcommand{\brwPst}{P^*_{\brwCh,\brwD}}
\newcommand{\brwEst}{E^*_{\brwCh,\brwD}}

\newcommand{\hdT}{\mathcal T}
\newcommand{\hdTi}[1]{\mathcal T_{#1}}
\newcommand{\hdNd}[1]{u_{#1}}
\newcommand{\hdX}[1]{v_{#1}}
\newcommand{\hdP}{\mathbf P_{\brwCh,\brwD}}
\newcommand{\hdE}{\mathbf E_{\brwCh,\brwD}}
\newcommand{\hdSh}{\sigma}
\newcommand{\hdR}[2]{R[#1,#2]}
\newcommand{\hdPplus}{\hdP^I}
\newcommand{\hdEplus}{\hdE^I}
\newcommand{\Cfconst}{C_G}
\newcommand{\spine}[1]{\mathcal S_#1}
\title{Capacity of the range of tree-indexed random walk}
\author{Tianyi Bai and Yijun Wan}
\begin{document}
\maketitle

\begin{abstract}
By introducing a new measure for the infinite Galton-Watson process and {providing} estimates for (discrete) Green's functions on trees, we {establish the asymptotic behavior} of the capacity of critical branching random walks: in high dimensions $d\ge 7$, the capacity grows linearly; and in the critical dimension $d=6$, it grows asymptotically {proportional to} $\frac{n}{\log n}$. 
\end{abstract}

\section{Introduction}
Given a probability distribution $\rwD$ on $\mathbb Z^d\,(d\ge 3)$, the \emph{capacity} of a finite set $A\subset\mathbb Z^d$ (with respect to $\rwD$) is defined as
\[
\capa A{:=}\sum_{x\in A}\rwP{\rwD}{x}(\tau^+_A=\infty),
\]
where $\rwP{\rwD}{x}$ {refers to} the law of a (discrete) random walk $(\rwS{n})$ started at $x$ with transition probability $\rwD$, and $\tau^+_A:=\inf\{n\geq 1: \rwS{n}\in A\}$ is {$(S_n)$'s} first returning time to $A$.

Let $\brwCh$ be a probability distribution on $\mathbb N$, and $\brwD$ be a probability distribution on $\mathbb Z^d$. Consider {the process that starts with a particle at $0\in\mathbb Z^d$. At each step, the particles die after generating a random number of new particles independently according to the law $\brwCh$, then these new particles drift away from their precursor independently according to the law $\brwD$.} This process is called \emph{branching random walk}, whose distribution is denoted by $\brwP$. The branching random walk is called \emph{critical} if $\brwCh$ has mean $1$, in which case, it is {well-known} that the process dies out in finite time almost surely (except for the trivial case that $\brwCh$ is the Dirac measure at $\{1\}$). 
The \emph{range} $R$ of this process, i.e. the set of points in $\mathbb Z^d$ visited by the branching random walk, is then almost surely finite. 
Moreover, we denote by $\{\#T=n\}$ the event that the branching random walk generates exactly $n$ particles in total before dying out. The notation $T$ actually stands for the genealogy tree of the process, see \Cref{preee} for details.

In this paper, we study the capacity of the range of {critical branching random walks in dimensions larger or equal to $6$, denoted by} $\capa R$, {conditioned on the event $\{\#T=n\}$ as $n\rightarrow\infty$}.

Throughout the paper, we shall consider distributions {$\brwCh$ on $\mathbb N$ and $\brwD,\rwD$ on $\mathbb Z^d$} with the assumptions
\begin{equation}\label{assumption}
\begin{aligned}
\left.
\begin{array}{lll}
{\mu\text{ has}}\text{ mean }1\text{ and finite variance, }{\text{and } \brwCh\not\equiv\delta_1},\\
\brwD \text{ is symmetric, aperiodic and irreducible such that }\rwE{\brwD}{0}\bracks*{e^{\sqrt{|\rwS{1}|}}}<\infty,\\
\rwD \text{ is aperiodic, irreducible with mean 0 and finite }(d+1) \text{-th moment },
\end{array}
\right\}
\end{aligned}
\end{equation}
{where $\rwE{\brwD}{0}$ refers to taking expectation with respect to the random walk $(\rwS{i})$ started at $0$ with transition probability $\brwD$}.

\begin{theorem}\label{mainresult}
Let $\brwCh,\brwD,\rwD$ be probability distributions with the conditions in \eqref{assumption}.
\begin{enumerate}
\item
In dimension $d\ge 7$, there is a constant $C(d,\brwCh,\brwD,\rwD)>0$ such that
under $\brwP(\cdot|\#T=n)$, as $n\rightarrow\infty$,
\[
\frac{\capa R}{n}\rightarrow C(d,\brwCh,\brwD,\rwD)\text{ in probability.}
\]
\item
In dimension $d=6$, if $\brwCh$ has finite $5$-th moment, then under $\brwP(\cdot|\#T=n)$, as $n\rightarrow\infty$,
\[
\frac{\log n}{n}\capa R\rightarrow 2\Cfconst^{-1}\text{ in probability,}
\]
where
\[
\Cfconst=\frac{1}{4\pi^{6}\sqrt{\det\rwVar{\rwD}\det\rwVar{\brwD}}}
\pars*{\sum_{k=0}^\infty (k-1)k\brwCh(k)}
C_f,
\]
\[C_f=\gnE\bracks*{\int_1^e dt \int_{\mathbb R^6} dx \cdot\rwJ{\rwD}(B_t^{\brwD}+x)^{-4}\rwJ{\brwD}(x)^{-4}},\]
$\rwVar{\rwD},\rwVar{\brwD}$ are the covariance matrices of $\rwD,\brwD$ {respectively}, $\rwJ{(\cdot)}(x)=\sqrt{x\cdot \rwVar{(\cdot)}^{-1}x}$, and $B_t^{\brwD}$ is the Brownian motion in $\mathbb R^6$ with covariance matrix $\rwVar{\brwD}$.
\end{enumerate}
\end{theorem}

\begin{remark}
\begin{enumerate}
\item Aperiodicity and irreducibility for $\brwD$ and $\rwD$ are {assumed} for convenience of the proofs. {In fact} the same results {in \Cref{mainresult}} hold {for $\rwD$ and $\brwD$ without those} assumptions. 
\item For $d\ge 7$, the constant $C(d,\brwCh,\brwD,\rwD)$ is implicit. {We refer the reader to \Cref{hd_constant} for more details.}
\item {The finite variance of the offspring distribution $\brwCh$ is required} in \Cref{Green_sum} for the high dimensions $d\ge 7$, and {the finite $5$-th moment of $\brwCh$ is required} in \Cref{n-tree} for the critical dimension $d=6$.
\item For the displacement law $\brwD$, the moment assumption
is required {for} the dyadic coupling in \Cref{dyadic_coupling}, {and the} symmetry is required {for} the conversion from our infinite model to finite trees, see \Cref{symmetry} for details. (We use the symmetry of $\brwD$ a few times elsewhere for convenience, but they are not essential.)
\item For the random walk distribution $\rwD$, the moment assumptions are required {for} the asymptotic estimates of Green's functions in \Cref{Green_asymptotic}.
\item If $\brwCh$ is the geometric distribution with parameter $\frac{1}{2}$, i.e. $\brwCh(k)=2^{-k-1}$, then $\brwP(\cdot|\#T=n)$ is the law of the random walk indexed by a uniformly chosen tree of $n$ nodes {considered in} \cite{LeGall-Lin-range}. 
In this case, by \Cref{main} and the methods developed in \cite[Section 3.1]{LeGall-Lin-range}, the convergence in probability for dimension $6$ holds in $L^2$-sense.
\item
If $\brwCh$ is the geometric distribution with parameter $\frac{1}{2}$,  $\brwD$ and $\rwD$ are one-step distributions of independent simple random walks, then $\Cfconst=9\pi^{-3}$. We refer the reader to \Cref{prop:srwC} for explicit calculations.
\end{enumerate}
\end{remark}
Historically, the study of the capacity of the range of simple random walks dates back to Jain and Orey \cite{JO69}, where a law of large numbers was established for $d\ge 3$. Then useful tools were developed in the book of Lawler \cite{Lawler-book-intersections}. Recently, numerous studies for the sharper estimates of the capacity appear in Chang \cite{Chang-} for $d=3$ (scaled convergence in distribution), Asselah, Schapira and Sousi \cite{Asselah-Schapira-Sousi-capRW} for $d\ge 6$, \cite{ASS4} for $d=4$, and Schapira \cite{Sc19} for $d=5$ (central limit theorem). 

If, in the definition of capacity, we simply replace the escape probability by $1$, then it gives us (the size of) the range $\#R$, which is a classical object for random walks, widely studied {since the work of} Dvoretzky and Erd\H{o}s \cite{DE51}, in which a 
law of large numbers was given for random walks in dimension $d\ge 1$. The corresponding central limit theorem was given by Jain and Orey \cite{JO69} for $d\ge 5$, Jain and Pruitt \cite{JP71} for $d\ge 3$, and Le Gall \cite{LG86} for $d\ge 2$. See also \cite{LGR91} for a general study of random walks in the domain of attraction of a stable distribution (i.e. without finite variance) by Le Gall and Rosen. 

For branching random walks, the law of large numbers for (the size of) its range $\#R$ was given by Le Gall and Lin in \cite{LeGall-Lin-range},\cite{LeGall-Lin-lowdim} for every $d\ge1$, where in the critical dimension $d=4$ they restrict to the geometric offspring distribution case. This result (in $d=4$) was then generalized by Zhu in \cite{zhu-cbrw} for general distributions. See also \cite{LZ10}, \cite{LZ11} for a related topic of local times of branching random walks. 

We summarize that, in view of law of large numbers, the critical dimension (the largest dimension with sublinear growth) is $d=2$ { for the range of the simple random walk (SRW) \cite{DE51}, $d=4$ for the range of the branching random walk (BRW) \cite{LeGall-Lin-range}, also $d=4$ for the capacity of the SRW \cite{JO69}, and $d=6$ for the capacity of the BRW.} 

Indeed, {the SRW or the BRW can be seen as a sequence of vertices, and one can establish corresponding infinite models for them with translational invariance property, which for the SRW started at $0$ is simply}
\[
(S_i)_{i\in\mathbb Z}\overset{d}{=}(S_{m+i}-S_m)_{i\in\mathbb Z}.
\]

Intuitively, this property shows that the SRW (or the BRW) is homogeneous in time. Moreover, either the range or the capacity can be decomposed into the sum over $i$ of the contribution of $\rwS{i}$, therefore, it boils down to a one-point estimate and a second moment estimate for its concentration property. One can express this one-point estimate in terms of Green's functions, and study Green's functions by moment estimates with a careful analysis {of} the tree (in the case of BRW) and the underlying random walk.

The rest of the paper is organised as follows. In \Cref{preee} we introduce {the models and some preliminary results regarding the capacity, Green's functions and the Brownian motion. The study of capacities of BRWs in high dimensions $d\ge 7$ is discussed in \Cref{highDimSection}, and {the case} of critical dimension $d=6$ is discussed in \Cref{TheCriticalDimension}. Finally in \Cref{lowDim}, we discuss two related open problems.} 
In particular, the main model with translational invariance property is established in \Cref{hd_model}, and the strategy with which we relate Green's functions {to the capacity} is showed in \Cref{sec_cap}. The behavior of Green's functions {is} mainly summarized in \Cref{Green_sum} and \Cref{core}. Finally, the two parts of \Cref{mainresult} are proved in \Cref{HighDimConclusion} and \Cref{mainpart2} respectively.

In the sequel, with a slight abuse of notations, each time we write a constant~$C(*)$, where $*$ is the set of parameters that this constant depends, it is only used in the current paragraph.

\section{Preliminaries}\label{preee}
In this section, we present systematically the definitions and models in this paper. 
\subsection{Trees and spatial trees}\label{GWdef}
A tree is a set $\gnT\subset\cup_{n\ge 0}\mathbb N^n_+$, such that
 \begin{itemize}
     \item The root $\varnothing\in \gnT$, where by convention we denote $\mathbb N_+^0=\{\varnothing\}$.
     \item If a node $\gnNd=(\gnNd_1,\dots,\gnNd_n)\in \gnT$, 
     then its parent $\overleftarrow{\gnNd}:=(\gnNd_1,\dots,\gnNd_{n-1})\in \gnT$. 
     \item For each node $\gnNd=(\gnNd_1,\dots,\gnNd_n)\in \gnT$, 
     there exists an integer $\gnCh{\gnNd}(\gnT)\ge0$, which is {the number of offspring of $u$ in $T$}, such that for every $j\in\mathbb N, (\gnNd_1,\dots,\gnNd_n,j)\in \gnT$ if and only if $1\le j \le \gnCh{\gnNd}(\gnT)$.
 \end{itemize}

We say that $\gnNd=(\gnNd_1,\dots,\gnNd_n)\in \gnT$ is an ancestor of $\gnNd'=(\gnNd_1',\dots,\gnNd_{n'}')\in\gnT$ if $n<n'$ and $\gnNd_i=\gnNd_i',\,1\le i\le n$, {and if this is the case, we will write $\gnNd\prec\gnNd'$.}
We also define the height (generation) of a node to be its length as a word, 
i.e. if $\gnNd=(\gnNd_1,\dots,\gnNd_n)$, then $|\gnNd|=n$. 
Moreover, we denote by $\#\gnT$ the total number of nodes. 
{In the following, we will omit $\gnT$ if it is clear that to which tree the nodes belong to from the context.} 

{Since nodes of $\gnT$ are sequences of natural numbers, there exists a natural lexicographical order for them. We can therefore explore $T$ in lexicographic order}
\[
\hdNd{0}=\varnothing,\hdNd{1},\hdNd{2},\dots.
\]
We remark that each node appears exactly once in this sequence if the tree is finite, thus if $\#\gnT=n$, the sequence terminates at $\hdNd{n-1}$.

Consider each node as a vertex, and add an edge between a node and its parent, then one can see $\gnT$ as an abstract graph. If we attach a vector  $\gnD{\gnNd}$ in $\mathbb Z^d$ to each directed edge $(\overleftarrow{\gnNd},\gnNd)$, fix the position of the root at $\gnX{\varnothing}=0$ and let $\gnX{\gnNd}=\sum_{\gnNd'\preceq\gnNd}\gnD{\gnNd'}$, then $(\gnX{\gnNd})_{\gnNd\in\gnT}$ gives a spatial tree tructure.

Given a distribution $\brwCh$ on $\mathbb N$ and a distribution $\brwD$ on $\mathbb Z^d$, we can define a probability measure on (spatial) trees, denoted by $\brwP$, under which we have that
\[
\gnCh{\gnNd}\overset{i.i.d.}{\sim}\brwCh,\,\gnD{\gnNd}\overset{i.i.d.}{\sim}\brwD.
\]
The abstract tree $\gnT$ under this law is called the \emph{Galton-Watson tree}, while the spatial tree $(\gnX{\gnNd})_{\gnNd\in\gnT}$ is called the \emph{branching random walk}. 

\subsection{The infinite model}\label{hd_model}
In this section, we {construct} an infinite model based on Galton-Watson trees that {will be used throughout this article and may be of independent interests to other problems}.
Intuitively, it can be seen as the {discrete limit} of critical Galton-Watson trees conditioned to be large (\cite[Section 2.6]{aldous-ii}),
and our construction generalises the {one-sided} version {of infinite Galton-Watson trees} in \cite[Section 2.2]{LeGall-Lin-range}. 

We define a \emph{forest indexed by a spine} to be a sequence of trees, (here $\hdTi{i}$ are {standard trees as in \Cref{GWdef})},
\[\hdT=((0,\hdTi{0}),(1,\hdTi{1}),(1,\hdTi{-1}),(2,\hdTi{2}),(2,\hdTi{-2})\dots),\]
where the roots $(\pm i,\varnothing)$ of $\hdTi{i}$ and $\hdTi{-i}$ $(i>0)$ are identified (glued together) as one single point {on the spine}.
We write $\gnCh{(i,\gnNd)}(\hdT)=\gnCh{\gnNd}(\hdTi{i})$ for the number of offspring of node $\gnNd\in\hdTi{i}$, 
and in particular, $\gnCh{(i,\varnothing)}^{+}(\hdT),\gnCh{(i,\varnothing)}^{-}(\hdT)$ are the numbers of offspring of points $(\pm i,\varnothing)$ in the two trees $\hdTi{i},\hdTi{-i}$, respectively.
We call the set of points $\{(i,\varnothing),i\in\mathbb N\}$ the \emph{spine} of $\hdT$, and $(0,\varnothing)$ the \emph{base point}. {Notice that by adding edges between consecutive points on the spine, the forest can also be seen as an abstract tree and the base point does not always take the role of the 'root', see \Cref{symmetry}.}

{We embed this forest in $\mathbb Z^d$, by taking $\gnD{(i,\gnNd)}(\hdT)=\gnD{\gnNd}(\hdTi{i})$ as the spatial displacement from its parent, and letting $\gnX{(i,\gnNd)}(\hdT)$ be the spatial position of $u$ by summing over all displacements along the path from the base point $(0,\varnothing)$ to $(i,\gnNd)$.}

On the set of forests, we define the following probability measure $\hdP$:
\begin{itemize}
\item Offspring distributions are independent, except for the two offspring distributions {of} the same node, $\gnCh{(i,\varnothing)}^{\pm}(\hdT)$. For each $i\ge 0,\gnNd\ne\varnothing$,
\[\gnCh{(i,u)}(\hdT)\overset{i.i.d.}{\sim}\brwCh,\]
moreover,
\[\gnCh{(0,\varnothing)}(\hdT){\sim}\brwCh,\]
while for other nodes $(\pm i,\varnothing)\,(i>0)$ on the spine
\[\hdP(\gnCh{(i,\varnothing)}^+(\hdT)=i,\gnCh{(i,\varnothing)}^-(\hdT)=j)=\brwCh(i+j+1).\]
\item Displacements $\gnD{(i,\gnNd)}(\hdT)$ are i.i.d. distributed as $\brwD$ on each directed edge including {edges on} the spine, with the base point fixed at the origin, $\gnX{(0,\varnothing)}(\hdT)=0$.
\end{itemize}
\begin{remark}
The law {of} the spine is indeed well-defined as a probability measure, because $\sum_{i,j\ge 0}\brwCh(i+j+1)=\sum_{k\ge 0}k\brwCh(k)=1$ for a critical distribution $\brwCh$.
\end{remark}

\begin{figure}[ht]
\centering
\includegraphics[scale=0.8]{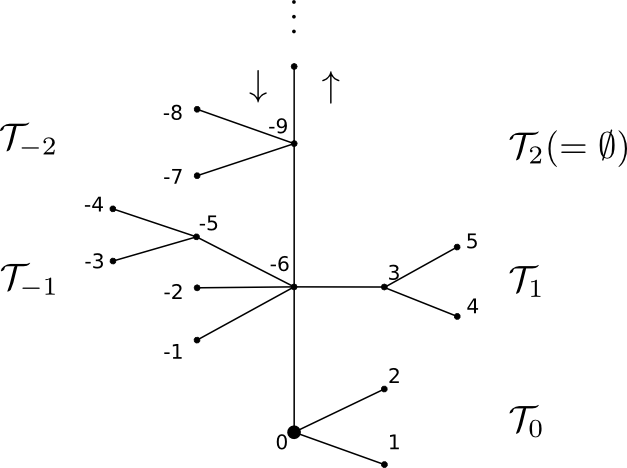}
\caption{Lexicalgraphical order on the forest indexed by spine.}
\label{fig-lex}
\end{figure}

{The lexicographical order of nodes on the forest is illustrated in \Cref{fig-lex}.} We denote this sequence (seen as vertices on a graph) by
\[
\dots,\hdNd{-1}(\hdT),\hdNd{0}(\hdT)=(0,\varnothing),\hdNd{1}(\hdT),\dots,\hdNd{n}(\hdT),\dots,
\]
and the corresponding spatial positions $(\gnX{\hdNd{i}}(\hdT))$ by
\begin{equation}\label{range_tree}
\dots,\hdX{-1}(\hdT),\hdX{0}(\hdT)=0,\hdX{1}(\hdT),\dots,\hdX{n}(\hdT),\dots.
\end{equation}
The range is defined as
\[
\hdR{i}{j}(\hdT)=\{\hdX{i}(\hdT),\hdX{i+1}(\hdT),\dots,\hdX{j}(\hdT)\}.
\]

\begin{figure}[ht]
\centering
\includegraphics[scale=0.7]{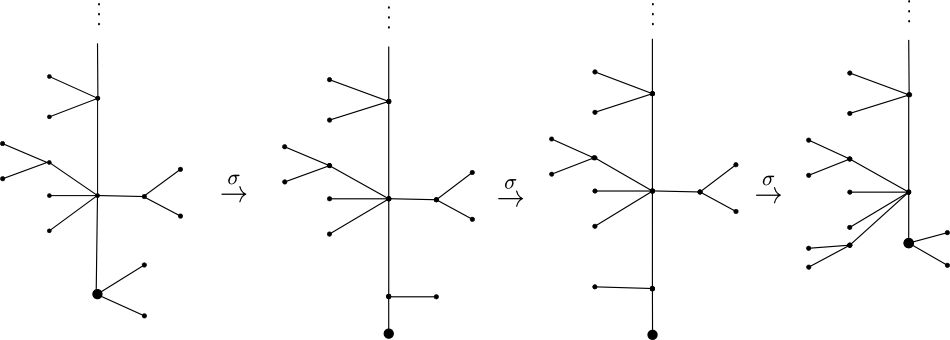}
\caption{The transform $\hdSh$ on the tree. Base points $(0,\varnothing)$ are marked with bigger circles.}
\label{fig-232}
\end{figure}

On the set of spine-indexed forests, we can then establish a shift transformation $\hdSh$ defined by (see \Cref{fig-232}):
\begin{equation}\label{hdSh}
\hdNd{i}(\hdSh(\hdT))=\hdNd{i+1}(\hdT),\, \hdX{i}(\hdSh(\hdT))=\hdX{i+1}(\hdT)-\hdX{1}(\hdT).
\end{equation}
One can easily check that $(\hdNd{i}(\hdSh(\hdT)))_{i\in\mathbb Z}$ is the same sequence as
$(\hdNd{i+1}(\hdT))_{i\in\mathbb Z}$,
and {$(\hdX{i}(\hdSh(\hdT)))_{i\in\mathbb Z}=(\hdX{i+1}(\hdT)-\hdX{1}(\hdT))_{i\in\mathbb Z}$ is the corresponding positions of $(\hdNd{i}(\hdSh(\hdT)))_{i\in\mathbb Z}$} in $\mathbb Z^d$,  translated such that the base point $(0,\varnothing)$ stays at the origin. Moreover, the transformation is invariant under $\hdP$. In other words, for any measurable set $A$ of spine-indexed forests,
\[
\hdP(\hdT\in A)=\hdP(\hdSh(\hdT)\in A).
\]
\begin{proposition}\label{hd_invariant_shift}
Given the assumption \eqref{assumption}, the probability measure $\hdP$ is invariant and ergodic under $\hdSh$. 
Consequently, we have that
\begin{equation}\label{invariant_hd}
(\hdX{i},\dots,\hdX{n+i})-\hdX{i}\overset{d}{=}(\hdX{0},\dots,\hdX{n})\text{  under  }\hdP,\,\forall i\in\mathbb Z,n\in\mathbb N.
\end{equation}
{In other words,}
\[
\hdR{i}{n+i}-\hdX{i}\overset{d}{=}\hdR{0}{n}\text{  under  }\hdP,\,\forall i\in\mathbb Z,n\in\mathbb N.
\]
\end{proposition}
\begin{proof}
Since $\brwD$ is symmetric, it suffices to study the abstract tree structure $(u_i)$. 

As {shown} in \Cref{fig-233}, take any node $\gnNd$: if it is the base point or some point not on the spine, then it has $k$ children (thus degree $k+1$) with probability $\brwCh(k)=\brwCh(\text{deg}(u)-1)$; otherwise, it has $i$ children on the left and $j$ children on the left (thus degree $i+j+2$) with probability $\brwCh(i+j+1)=\brwCh(\text{deg}(u)-1)$. 

\begin{figure}[ht]
\centering
\includegraphics[scale=1]{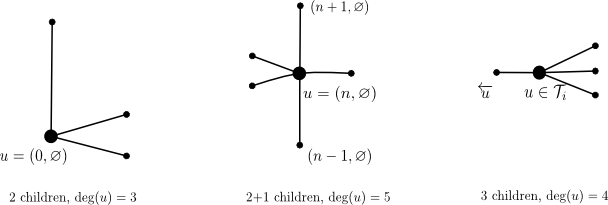}
\caption{Neighborhood of a single node. Degree means the number of adjacent nodes as in an abstract graph.}
\label{fig-233}
\end{figure}

Therefore, $\hdP$ can be seen as a probability measure on spine-indexed forests such that each node $\gnNd$ has degree $k+1$ with probability $\brwCh(k)$. {That is to say, $\hdP$ only takes into account the abstract tree structure,} regardless of the base point. For example, denote by $t$ and $t'$ the structures depicted in \Cref{sec234}, and by $A$ and $A'$ the cylinder sets of forests whose first two or three subtrees are identical to $t$ and $t'$ respectively, then
\begin{align*}
\hdP(\hdT\in A)
=\prod_{\gnNd\in t}\brwCh(\text{deg}(\gnNd)-1)
=\prod_{\gnNd\in t'}\brwCh(\text{deg}(\gnNd)-1)
=\hdP(\hdT\in A').
\end{align*}

\begin{figure}[ht]
\centering
\includegraphics[scale=1]{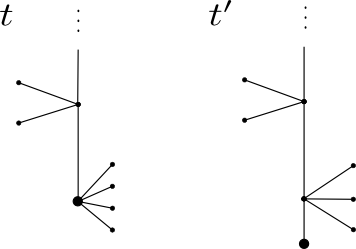}
\caption{Finite trees $t$ and $t'$ that are only different in the position of base points.}
\label{sec234}
\end{figure}

Since $\hdSh$ only changes the base point, it is then invariant with respect to $\hdP$. Ergodicity is also clear by construction.

{Then  \eqref{invariant_hd} follows} easily by applying the invariant transform as illustrated below:
\begin{align*}
 &\hdP(\hdX{2}(\hdT)-\hdX{1}(\hdT)=x)\\
=&\hdP(\hdX{1}(\hdSh(\hdT))-\hdX{0}(\hdSh(\hdT))=x)\\
=&\hdP(\hdX{1}(\hdT)-\hdX{0}(\hdT)=x),
\end{align*}
where we use the invariance property of $\hdSh$ with respect to $\hdP$ in the last line.
\end{proof}
\begin{remark}\label{half_model}
If one is only interested in the positive side, 
\begin{equation*}
((0,\hdTi{0}),(1,\hdTi{1}),(2,\hdTi{2}),\dots),
\end{equation*}
then the spine has offspring distribution
\[\hdP(\gnCh{(i,\varnothing)}(\hdT)=i)=\sum_{j=0}^\infty\hdP(\gnCh{(i,\varnothing)}^+(\hdT)=i,\gnCh{(i,\varnothing)}^-(\hdT)=j)=\sum_{j=0}^\infty\brwCh(i+j+1),\]
which {is consistent with the construction in \cite[Section 2.2]{LeGall-Lin-range}, for which the invariant transformation can be also induced by transformation $\hdSh$ defined in \eqref{hdSh}}.
\end{remark}

\begin{remark}
If we are interested in trees with $n$ nodes instead of infinite nodes, with the same spirit as in the proof of \Cref{hd_invariant_shift}, one has the equivalence between {Galton-Watson trees conditioned on total population size $=n$} and \emph{simply generated trees} in \cite[Section 2.1]{aldous-ii}. For a tree with $n$ nodes, one has to specify a root (both for the branching process and {the} combinatoric model), while in the infinite case, the 'root' is naturally set {at} infinity, and the 'base point' is actually redundant (for the combinatoric model).
\end{remark}
\begin{remark}\label{symmetry}
If we replace edges in our model by directed edges of distribution $\brwD$ pointing towards infinity, then
\Cref{hd_invariant_shift} still holds without assuming that $\brwD$ is symmetric. 

In contrast, {the} standard branching random walk with asymmetric displacement {is constructed by attaching displacements to the directed edges of the Galton-Watson tree pointing towards the root.}

\begin{figure}[ht]
\centering
\includegraphics[scale=1]{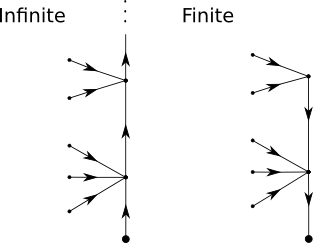}
\caption{Directed edges for the infinite and finite models. Directions of edges are different on the 'spine'.}
\label{sec235}
\end{figure}

Therefore, for asymmetric $\brwD$, the role of the base point $(0,\varnothing)$ here and the role of the root in {the} standard branching random walk are different, and we can no longer compare them by identifying the base point of the infinite model as the root of a standard finite model, which is the method in \Cref{new_hd_2}. The displacement distribution $\brwD$ is thus assumed symmetric.
\end{remark}

\subsection{Estimates on random walks and Green's functions}\label{Lawler-Green-estimate}

In this section, we present a few estimates on random walks and Green's function. We {denote by $\rwP{\rwD}{x}$ the law of the random walk started at $x$ with transition probability $\rwD$, and by $(\rwS{n})$ the random walk under $\rwP{\rwD}{x}$ (or $\rwS{n}^{(i)}$ for its i.i.d. copies). Then the $\rwD$-Green's function is defined as}
\[
\green{\rwD}(x,y)=\green{\rwD}(x-y)
=\sum_{n=0}^\infty\rwP{\rwD}{0}(\rwS{n}=x-y).
\]

\begin{lemma}\label{LCLT}
\cite[p.24]{Lawler-book-RW} 
Let $\rwD$ be an aperiodic and irreducible distribution on $\mathbb Z^d\,(d\ge 1)$ with mean $0$ and finite third moment. Denote by $\rwVar{\rwD}$ the covariance matrix of $\rwD$. Then there exists a constant $C(d,\rwD)>0$ such that, uniformly for all $x\in\mathbb Z^d$,
\begin{align*}
&\absolute*{\rwP{\rwD}{0}(\rwS{n}=x)-\frac{1}{(2\pi n)^{d/2}\sqrt{\det\rwVar{\rwD}}}e^{-\frac{x\cdot\rwVar{\rwD}^{-1}x}{2n}}}\le C(d,\rwD){n^{-\frac{d+1}{2}}}
\end{align*}
\end{lemma}

\begin{lemma}\cite[Theorem 4.3.5]{Lawler-book-RW}
\label{Green_asymptotic}
Given an aperiodic and irreducible distribution $\rwD$ on $\mathbb Z^d(d\ge 3)$ with mean $0$ and covariance matrix $\rwVar{\rwD}$,
if it has finite $(d+1)$-th moment $\rwP{\rwD}{0}(|\rwS{1}|^{d+1})<\infty$, then
\[
\green{\rwD}(x)=\frac{C_{d,\rwD}}{\rwJ{\rwD}(x)^{d-2}}+O(|x|^{1-d}),
\]
where 
$C_{d,\rwD}=\frac{\mathbf \Gamma(\frac{d}{2})}{(d-2)\pi^{d/2}\sqrt{\det\rwVar{\rwD}}},$
$\mathbf \Gamma(\cdot)$ refers to the Gamma function and $\rwJ{\rwD}(x)=\sqrt{x\cdot\rwVar{\rwD}^{-1}x}.$
\end{lemma}

\begin{lemma}\label{A.1}
Let $\rwD$ be an aperiodic and irreducible distribution on $\mathbb Z^d\,(d\ge 3)$ with mean $0$ and finite third moment and $1\le m\le d-1$. There exists a constant $C(d,\rwD)>0$ such that uniformly on the starting point $x_0\in\mathbb{Z}^d$, 
\[\rwE{\rwD}{x_0}(|\rwS{n}|\vee 1)^{-m}\le C(d,\rwD)n^{-\frac{m}{2}}.\]
\end{lemma}
\begin{proof}
Due to irreducibility of $\rwD$, we have $\rwJ{\rwD}(x)^2\ge C_1(d,\rwD)|x|^2$.
Then by \Cref{LCLT}, we can find $C_2(d,\rwD)>0$ such that
\begin{align*}
\rwE{\rwD}{x_0}(|\rwS{n}|\vee 1)^{-m}
&\le C_2(d,\rwD)\sum_{x\in\mathbb Z^d}(|x_0+x|\vee 1)^{-m}{n^{-\frac d 2}}e^{-\frac{C_1(d,\rwD)|x|^2}{2n}}+O(n^{-\frac{d+1}{2}})\\
&\le C_2(d,\rwD)n^{-\frac{m}{2}}
\sum_{x\in\mathbb Z^d/\sqrt{n}}
\pars*{\absolute*{\frac{x_0}{\sqrt{n}}+x}\vee\frac{1}{\sqrt n}}^{-m}
{n^{-\frac d 2}}e^{-\frac{C_1({d,\rwD})|x|^2}{2}}+O(n^{-\frac{d+1}{2}}).
\end{align*}
Moreover, denote by $B(y;r)$ the ball centered at $y$ with radius $r$, then
\begin{align*}
&\quad\ {n^{-\frac d 2}}\sum_{x\in \mathbb Z^d/\sqrt{n}}\left(\left|\frac{x_0}{\sqrt{n}}+x\right|\vee\frac{1}{\sqrt{n}}\right)^{-m}e^{-\frac{C_1({d,\rwD})|x|^2}{2}} \\
&\le {n^{-\frac d 2}}\pars*{
\sum_{x\in(\mathbb Z^d/\sqrt{n})\cap B(-x_0/\sqrt n;1)}\left(\left|\frac{x_0}{\sqrt{n}}+x\right|\vee\frac{1}{\sqrt{n}}\right)^{-m}
+\sum_{x\in(\mathbb Z^d/\sqrt{n})\backslash B(-x_0/\sqrt n;1)}e^{-\frac{C_1({d,\rwD})|x|^2}{2}}}\\
&\underset{n\rightarrow\infty}{\longrightarrow} \int_{B(0;1)}|x|^{-m}dx + \int_{\mathbb R^d}e^{-\frac{C_1({d,\rwD})|x|^2}{2}}dx.
\end{align*}
which is a constant depending only on~$d$ and $\rwD$.
\end{proof}

\begin{corollary}\label{A.2}
Let $\rwD$ be an aperiodic and irreducible distribution on $\mathbb Z^d\,(d\ge 3)$ with mean $0$ and finite third moment, then for any $m\ge 1$, 
\begin{enumerate}
\item there exists a constant $C(d,\rwD,m)>0$ such that uniformly {for} $x_0\in\mathbb Z^d$,
\[
\rwE{\rwD}{x_0}\bracks*{
\pars*{
\sum_{i=0}^n (|\rwS{i}|\vee 1)^{-2}
}^{m}
}
\le C(d,\rwD,m)(\log n)^m;
\]
\item for any $k>2$, there exists a constant $C'(d,\rwD,m,k)>0$ such that uniformly {for} $x_0\in\mathbb Z^d$,
\[
\rwE{\rwD}{x_0}\bracks*{
\pars*{
\sum_{i=0}^n (|\rwS{i}|\vee 1)^{-k}
}^{m}
}
\le C'(d,\rwD,m,k).
\]
\end{enumerate}
\end{corollary}
\begin{proof}
The cases $m=1$ for both $k=2$ and $k>2$ are clear by \Cref{A.1}. For $m\ge 2$, applying Markov's property inductively gives that
\begin{align*}
&\rwE{\rwD}{x_0}\bracks*{
\pars*{
\sum_{i=0}^n (|\rwS{i}|\vee 1)^{-k}
}^{m}
}\\
{\le} &{C'(d,\rwD,m,k)}
\rwE{\rwD}{x_0}\bracks*{
\sum_{i=0}^n
\pars*{
(|\rwS{i}|\vee 1)^{-k}
}
\cdot
\rwE{\rwD}{S_i}\bracks*{
\pars*{
\sum_{j=0}^{n-i} \pars*{|S'_{j}|\vee 1}^{-k}
}^{m-1}}
},
\end{align*}
where $\pars*{S'_{j}}$ denotes a random walk independent of $(\rwS{i})$.
\end{proof}

\begin{lemma}\label{dyadic_coupling}\cite[Theorem 4]{extension-kmt}
Let $\rwD$ be a probability distribution in $\mathbb R^d$ with mean $0$ and covariance matrix $\rwVar{\rwD}$. If $\rwE{\rwD}{0}\bracks*{e^{\sqrt{|\rwS{1}|}}}<\infty$,
then one can construct on the same probability space a Brownian motion $(B_t)$ with covariance matrix $\rwVar{\rwD}$ such that there exists $C,C'>0$ depending on $d,\rwD$ such that
\[
\rwP{\rwD}{0}\pars*{\max_{1\le k\le n}\absolute*{\rwS{k}-B_k}\ge x}\le \frac{Cn}{e^{C'\sqrt{x}}}.
\]
\end{lemma}
\subsection{Capacity}\label{sec_cap}
Given a distribution $\rwD$ on $\mathbb Z^d\,(d\ge 3)$ and a finite set $A\subseteq\mathbb Z^d$, recall that the $\rwD$-capacity is defined as
\begin{equation}\label{cap_def}
\capa A=\sum_{x\in A}\rwP{\rwD}{x}(\tau_A^+=\infty).
\end{equation}

In this section, we give two estimates relating {the $\rwD$-capacity to the $\rwD$-Green's function,} which is defined as
\[
\green{\rwD}(x,y)=\green{\rwD}(x-y)
=\rwE{\rwD}{0}\bracks*{\sum_{i=0}^\infty\mathbf 1_{(\rwS{i}=x-y)}}
=\sum_{i=0}^\infty\rwP{\rwD}{0}(\rwS{i}=x-y),\,x,y\in\mathbb Z^d.
\]

\begin{lemma}\label{capA}
Let $d\ge 3$ {and} $\rwD$ be any probability distribution on $\mathbb Z^d$. For any finite set $A\subset \mathbb Z^d$ and $k\in\mathbb N_+$, 
\[\capa A\ge\frac{\#A}{k+1}-\frac{\sum_{x,y\in A}\green{\rwD}(x,y)}{k(k+1)}.\]
\end{lemma}
\begin{proof}
We define local times
$L_A:=\sum_{n=1}^\infty\mathbf 1_{(\rwS{n}\in A)}\in\mathbb N\cup\{\infty\}$ for any finite set $A\subset\mathbb Z^d$, then by definition, 
$\capa A=\sum\limits_{x\in A}\rwP{\rwD}{x}(L_A=0).$ 

For any integers $a>0$ and $b\ge 0$,
\begin{align*}
\sum_{x\in A}\rwP{\rwD}{x}(L_A=a)\rwP{-\rwD}{x}(L_A=b)
    =&\sum_{x,y\in A}\rwP{\rwD}{x}(S_{\tau_A^+}=y)\rwP{\rwD}{y}(L_A=a-1)\rwP{-\rwD}{x}({L}_A=b)\\
    =&\sum_{x,y\in A}\rwP{-\rwD}{y}({S}_{\tau_A^+}=x)\rwP{\rwD}{y}(L_A=a-1)\rwP{-\rwD}{x}({L}_A=b)\\
    =&\sum_{y\in A}\rwP{\rwD}{y}(L_A=a-1)\rwP{-\rwD}{y}({L}_A=b+1),
    \end{align*}
    where $-\rwD$ refers to the distribution with $-\rwD(x):=\rwD(-x),\,\forall x\in\mathbb Z^d.$

Thus by induction we have that
\[\sum_{x\in A}\rwP{\rwD}{x}(L_A=a)\rwP{-\rwD}{x}({L}_A=b)=\sum_{x\in A}\rwP{\rwD}{x}(L_A=0)\rwP{-\rwD}{x}({L}_A=a+b).\]
By summing over $a\le k$ and $b\ge0$, it follows that
\begin{equation}\label{doub-L}
\begin{aligned}
    \sum_{x\in A}\rwP{\rwD}{x}(L_A\le k) &= \sum_{y\in A}\rwP{\rwD}{y}(L_A=0)\left(\sum_{a=0}^k\rwP{-\rwD}{y}({L}_A\ge a)\right)\\
    &\le
    (k+1)\sum_{y\in A}\rwP{\rwD}{y}(L_A=0).
\end{aligned}
\end{equation}
Therefore
\begin{align*}
\#A-\sum_{x\in A}\rwP{\rwD}{x}(L_A>k)
&=\sum_{x\in A}(1-\rwP{\rwD}{x}(L_A> k))\\
&=\sum_{x\in A}\rwP{\rwD}{x}(L_A\le k)\\
&\le (k+1)\sum_{y\in A}\rwP{\rwD}{y}(L_A=0)=(k+1)\capa A.
\end{align*}
To conclude, it suffices to notice that $\rwP{\rwD}{x}(L_A>k)\le\frac{\sum_{y\in A}\green{\rwD}(x,y)}{k}$, which follows directly from Markov's inequality.
\end{proof}

Moreover, in our situation, the set $A_n=\{\xn{0},\dots,\xn{n}\}$ is the trajectory of a stationary process $(\xn{n})_{n\in\mathbb Z}$ up to translation (under some probability space $(\Omega,\mathcal F,\gnP)$), in the sense that
\begin{equation}\label{stationary_process}
\xn{0}=0,\,\braces{\xn{0},\dots,\xn{n}}\overset{d}{=}\braces{\xn{i},\dots,\xn{n+i}}-\xn{i},\,\forall i\in\mathbb Z,\forall n\in\mathbb N,
\end{equation}
where $A-x:=\setof{a-x}{a\in A}$ for any set $A\subseteq \mathbb Z^d$ and $x\in \mathbb Z^d$.
We can thus rewrite \eqref{cap_def} as
\begin{equation}\label{decompose_capa}
\capa A_n=\sum_{i=0}^n\mathbf 1_{\braces{\xn{i}\not\in\braces{\xn{i+1},\dots,\xn{n}}}}\rwP{\rwD}{\xn{i}}(\tau_{A_n}^+=\infty).
\end{equation}
and take expectation to get
\begin{align*}
\gnE\capa A_n&=\sum_{i=0}^n\gnE\bracks*{\mathbf 1_{\braces{\xn{i}\not\in\braces{\xn{i+1},\dots,\xn{n}}}}\rwP{\rwD}{\xn{i}}(\tau_{\braces{\xn{0},\dots,\xn{n}}}^+=\infty)}\\
&=\sum_{i=0}^n\gnE\bracks*{\mathbf 1_{\braces{\xn{0}\not\in\braces{\xn{1},\dots,\xn{n-i}}}}\rwP{\rwD}{\xn{0}}(\tau_{\braces{\xn{-i},\dots,\xn{n-i}}}^+=\infty)}.
\end{align*}
This sum may be approximated (with a second moment method, for instance) by $n$ times 
\begin{equation}\label{quantity_capa}
\gnE\bracks*{\mathbf 1_{\braces*{\xn{0}\not\in\braces{\xn{1},\dots,\xn{\xi^r_n}}}}\rwP{\rwD}{\xn{0}}\pars*{\tau_{\braces{\xn{-\xi^l_n},\dots,\xn{\xi^r_n}}}^+=\infty}},
\end{equation}
where $\xi^l_n$ and $\xi^r_n$ are geometric killing times with parameter $\frac{1}{n}$.

The following lemma inspired by \cite[Theorem 3.6.1]{Lawler-book-intersections} then allows us to establish a relation between \eqref{quantity_capa} and Green's functions. {Recall} that $\xi$ is a geometric variable with parameter $\lambda$ if 
\[
\gnP(\xi=k)=\lambda(1-\lambda)^{k},\,k\in\mathbb N.
\]

\begin{lemma}\label{cap_relation}
Let $(\xn{n})_{n\in\mathbb Z}\in\mathbb Z^d$ be a stationary process up to translation in \eqref{stationary_process}. 
Let $d\ge 3,n\ge 1$, and let $\xi^l_n,\xi^r_n,\xi_n$ be independent geometric random variables with parameter $\frac{1}{n}$. 
If we set
\begin{align*}
&I_n=\mathbf 1_{\{\xn{0}\ne \xn{i},0<i\le \xi^r_n\}},\\
&E_n=\rwP{\rwD}{\xn{0}}\pars*{\tau^+_{\braces{\xn{-\xi^l_n},\dots,\xn{\xi^r_n}}}>\xi_n},\\
&G_n=\sum_{i=-\xi^l_n}^{\xi^r_n}\greenkill{\rwD}{1-\frac{1}{n}}(\xn{0},\xn{i}),
\end{align*}
where $
\greenkill{\rwD}{\lambda}(x)=\sum_{k\ge 0}\lambda^k\rwP{\rwD}{0}(\rwS{k}=x)
$ denotes the Green's function with killing rate $\lambda$,
then 
\[
\gnE [E_nG_nI_n]=1.
\]
\end{lemma}
\begin{proof}
For $m\in\mathbb N$ and $x_1,\dots,x_{m}$ in $\mathbb Z^d$, we consider the event
\[B=B(m;x_1,\dots,x_m):=\braces{\xi_n^l+\xi_n^r=m,\xn{i-\xi^l_n}=\xn{-\xi^l_n}+x_{i},\forall 0\le i\le m}\]
with the convention that $x_0=0$.
When $m$ runs through $\mathbb N$ and $(x_i)$ runs through all possible finite sequences of $\mathbb Z^d$, we have that
\[
\sum_{m\ge 0}\sum_{{x_1,\dots,x_m\in\mathbb Z^d}}\mathbf 1_{B(m;x_1,\dots,x_m)}=1.
\]
Therefore, it suffices to prove that for any $B=B(m;x_1,\dots,x_m)$,
\[
\gnE[\mathbf 1_{B}E_nG_nI_n]=\gnP(B).
\]

Moreover, on a fixed $B$, we can  define
\[
B_j=\braces{\xi_n^l=j,\xi_n^r=m-j,\xn{i}=\xn{0}+x_{i},\forall 0\le i\le m},\,0\le j\le m,
\]
then since $E_n,I_n,G_n$ are all invariant under the translation $(\xn{i})\rightarrow (\xn{i}-\xn{-\xi_n^l})$, we have that
\begin{align*}
\gnE[\mathbf 1_{B}E_nG_nI_n]
&=\sum_{j=0}^m\gnE[\mathbf 1_{B_j}E_nG_nI_n]\\
&=\sum_{j=0}^m\gnP(B_j)\mathbf 1_{\{x_j\ne x_i,j<i\le m\}}\rwP{\rwD}{x_j}\pars*{\tau^+_{\braces{x_{0},\dots,x_{m}}}>\xi_n}\sum_{k=0}^{m}\greenkill{\rwD}{1-\frac{1}{n}}(x_j,x_k).
\end{align*}
By the stationary property \eqref{stationary_process}, we have that
\[
\gnP(B_j)=\frac{\gnP(B)}{m+1},\,\forall 0\le j\le m,
\]
thus we can further simplify the equation above to
\begin{equation}\label{361_eq}
\gnE[\mathbf 1_{B}E_nG_nI_n]
=\frac{\gnP(B)}{m+1}\sum_{k=0}^m\sum_{j=0}^m\mathbf 1_{\{x_j\ne x_i,j<i\le m\}}\rwP{\rwD}{x_j}\pars*{\tau^+_{\braces{x_{0},\dots,x_{m}}}>\xi_n}\greenkill{\rwD}{1-\frac{1}{n}}(x_j,x_k).
\end{equation}

For any $A\subseteq \mathbb Z^{d},z\in A$, by decomposing the random walk $(\rwS{n})$ started at $z$ at the last time it hits $A\subseteq\mathbb Z^d$, it is {not hard to see} that (\cite[Proposition 2.4.1 (b)]{Lawler-book-intersections})
\[\sum_{x\in A}\rwP{\rwD}{x}(\tau^+_{A}>\xi_n)\greenkill{\rwD}{1-\frac{1}{n}}(z,x)=1.\]
Take $A=\braces*{x_0,x_1,\dots,x_m}$, {then} we have that
\[\sum_{j=0}^m \mathbf 1_{\braces{x_j\ne x_i,j<i\le m}}\rwP{\rwD}{x_j}(\tau^+_{A}>\xi_n)\greenkill{\rwD}{1-\frac{1}{n}}(z,x_j)=1,\,z\in \braces*{x_0,x_1,\dots,x_m}.\]
Put this into \eqref{361_eq}, then
\[
\gnE[\mathbf 1_{B}E_nG_nI_n]
=\frac{\gnP(B)}{m+1}\sum_{k=0}^m 1=\gnP(B).
\]
The conclusion follows by adding up all choices of $B(m;x_1,\dots,x_m)$.
\end{proof}

\subsection{Strong mixing property for functions of the Brownian motion}\label{sec_str}
The calculation for Green's functions will lead to some estimates of the following form, for which we give a concentration result in advance. This part is inspired partially by \cite[Lemma 18, Lemma 19]{LeGall-Lin-range}.

In this section, let $d\ge 3$, and we {consider a continuous homogeneous functions of degree $2$,} $f:\mathbb R^d\backslash\{0\}\rightarrow \mathbb R^+$ such that\begin{equation}\label{beginning}
f(\lambda z)=\lambda^{-2}f(z),z\in\mathbb R^d\backslash \{0\},\lambda\in\mathbb R\backslash \{0\}.
\end{equation}
Then in particular, $f$ is bounded on the unit sphere, and 
\begin{equation}\label{asympf}
f(z)\asymp |z|^{-2},\,z\rightarrow\infty.
\end{equation}

Consider the trajectory $w$ of a $d$-dimensional Brownian motion. Let
\[
F(w)=\int_1^e f(w(t))dt
\]
and 
\[
Tw(t)=\frac{w(et)}{\sqrt e}.
\]
Then one can easily deduce that:
\begin{enumerate}
\item $F$ is almost surely finite;
\item $T$ is invariant and ergodic in the Wierner space equipped with the probability measure of the Brownian motion;
\item $\int_1^{e^n}f(w(t))=F(w)+F(Tw)+\dots+F(T^{n-1}w).$
\end{enumerate}
Thus by Birkhoff's ergodic theorem, for the Brownian motion $(B_t)$ in $\mathbb R^d$, the following integral converges almost surely to its expectation,
\begin{equation}\label{birk}
\frac{\int_1^{e^n}f(B_t)dt}{n}\rightarrow \gnE\bracks*{\int_1^e f(B_t)dt}.
\end{equation}
Moreover, we can improve it to a concentration property,
\begin{proposition}\label{str_mixing}
Let $(B_t)$ be the Brownian motion in $\mathbb R^d\, (d\ge 3)$ with non-degenerate covariance matrix $\Gamma$.
Then for any $\epsilon>0$, $m>0$ and $f$ satisfying \eqref{beginning}, there exists a constant $C(d,\epsilon,m,\Gamma)>0$ such that
\begin{equation}\label{eq:str_mixing}
\gnP\pars*{\absolute*{\frac{\int_1^{n}f(B_t)dt}{\log n}-\gnE \bracks*{\int_1^e f(B_t)dt}}>\epsilon}\le C(d,\epsilon,m,\Gamma)(\log n)^{-m},\forall n\ge 1.
\end{equation}
\end{proposition}
To prove this, we need the following moment estimate,
\begin{lemma}\cite[Theorem 1]{strong_mixing}\label{str_mixing_moment}
Let $(X_n)_{n\in\mathbb Z}$ be a (strictly) stationary sequence, i.e. a sequence of random variables such that for any $k\in\mathbb N$ and $t,t_1,\dots,t_k\in\mathbb Z$
\[
\pars*{X_{t_1},\dots,X_{t_k}}\overset{d}{=}\pars*{X_{t_{1}+t},\dots,X_{t_k+t}}.
\]
Let $\mathcal M_i^j$ be the $\sigma$-field generated by $\braces{X_i,X_{i+1},\dots,X_j}$,
and let 
\[
\alpha(n)=\sup_{A\in\mathcal M_{-\infty}^0,B\in\mathcal M_n^\infty}\absolute*{\gnP(A\cap B)-\gnP(A)\gnP(B)}.
\]
For $r>2,\delta>0$, if
we have $\gnE X_1=0,\,\gnE|X_1|^{r+\delta}<\infty$ and
\[
\sum_{n=0}^\infty(n+1)^{\frac{r}{2}-1}\alpha(n)^{\frac{\delta}{r+\delta}}<\infty,
\]
then there exists a constant $C(r,\delta)>0$ such that
\[
\gnE|X_1+\dots +X_n|^r\le C(r,\delta)n^{\frac{r}{2}},\,\forall n\ge 1.
\]
\end{lemma}
\begin{proof}[Proof of \Cref{str_mixing}]
It suffices to prove {\eqref{eq:str_mixing} for the Brownian motion with covariance $\text{I}_d$.}
Let 
\[
X_n=\int_{e^{n-1}}^{e^n}f(B_t)dt-\gnE \bracks*{\int_1^e f(B_t)dt}=F(T^{n-1}B_t)-\gnE \bracks*{\int_1^e f(B_t)dt},
\]
then by a change variable from $n$ to $e^n$, it suffices to show that there exists $C(d,\epsilon,m)>0$ with
\[
\gnP\pars*{
\absolute*{
X_1+\dots+X_n
}>\epsilon n
}\le C(d,\epsilon,m)n^{-m}.
\]

By \eqref{asympf}, $(X_n)$ is a stationary sequence with mean $0$. Moreover, by applying the same trick as in \Cref{A.2}, we can easily show that it also satisfies the moment requirement $\gnE|X_1|^{r+\delta}<\infty$ for all $r,\delta$. Therefore, to apply \Cref{str_mixing_moment}, it suffices to prove that for $(X_n)$ we have $\alpha(n)=O(e^{-cn})$ for some $c>0$.

Since $(X_n)$ { only depends on the trajectory} of the Brownian motion, which is a Markov process, we have that
\[
\alpha(n)=\sup_{A,B\subseteq\mathbb R^d}\absolute*{\gnP(B_1\in A,B_{e^n}\in B)-\gnP(B_1\in A)\gnP(B_{e^n}\in B)}.
\]
Clearly, $\gnP(|B_1|>n)=\gnP(|B_{e^n}|>ne^{n/2})=O(e^{-cn}),$ so we may consider the supreme restricted to bounded balls in $\mathbb R^d$, $A\subseteq \text{Ball}(0;n),\,B\subseteq \text{Ball}(0;n e^{n/2})$.
Then we expand $\alpha(n)$ by definition,
\begin{align*}
&\sup_{\substack{A\subseteq\text{Ball}(0;n)\\B\subseteq\text{Ball}(0;ne^{n/2})}}
\absolute*{\gnP(B_1\in A,B_{e^n}\in B)-\gnP(B_1\in A)\gnP(B_{e^n}\in B)}\\
\le&\frac{1}{(2\pi)^d}
\sup_{\substack{A\subseteq\text{Ball}(0;n)\\B\subseteq\text{Ball}(0;ne^{n/2})}}
\int_A dx \int_B dy\cdot
\absolute*{
\frac{1}{\sqrt{e^n-1}}
e^{-\frac{|x|^2}{2}-\frac{|y-x|^2}{2(e^n-1)}}
-
\frac{1}{\sqrt{e^n}}
e^{-\frac{|x|^2}{2}-\frac{|y|^2}{2e^n}}
}
\\
\le
&\frac{1}{(2\pi)^d}
\sup_{\substack{A\subseteq\text{Ball}(0;n)\\B\subseteq\text{Ball}(0;ne^{n/2})}}
\int_A dx \int_B dy\cdot
\frac{1}{\sqrt{e^n}}
\absolute*{
e^{-\frac{|x|^2}{2}-\frac{|y-x|^2}{2(e^n-1)}}
-
e^{-\frac{|x|^2}{2}-\frac{|y|^2}{2 e^n}}}\\
&+\frac{1}{(2\pi)^d}
\sup_{\substack{A\subseteq\text{Ball}(0;n)\\B\subseteq\text{Ball}(0;ne^{n/2})}}
\int_A dx \int_B dy\cdot
\absolute*{
\frac{1}{\sqrt{e^n-1}}-\frac{1}{\sqrt{e^n}}}
e^{-\frac{|x|^2}{2}-\frac{|y-x|^2}{2(e^n-1)}}\\
\le&
\frac{1}{(2\pi)^de^{\frac n 2}}
\sup_{\substack{A\subseteq\text{Ball}(0;n)\\B\subseteq\text{Ball}(0;ne^{n/2})}}
\int_A 
e^{-\frac{|x|^2}{2}}dx
\int_B
e^{-\frac{|y|^2}{2 e^n}}dy\cdot O\pars*{\absolute*{\frac{|y-x|^2}{2(e^n-1)}-\frac{|y|^2}{2 e^n}}}+O(e^{-n})\\
=&O(n^2e^{-n}),
\end{align*}
where in the last line, we upper bound the integrals by $1$ and use that $|x|\le n,\,|y|\le ne^{n/2}$.

In conclusion, we have $\alpha(n)=O(e^{-cn})$ for some $c>0$, thus the {conditions in} \Cref{str_mixing_moment} are satisfied. Therefore, let $\delta=\frac{1}{2}$, then for any $r>2$, there exists a constant $C(r)$ such that
\[
\gnE|X_1+\dots +X_n|^r\le C(r)n^{\frac{r}{2}},\,\forall n\ge 1,
\]
then by a Chebyshev-type inequality,
\[
\gnP\pars*{|X_1+\dots+X_n|\ge k(C(r)n^{\frac{r}{2}})^{\frac{1}{r}}}\le k^{-r}.
\]
{The} conclusion follows by taking $r=2m$ and $k=\epsilon n^{\frac{1}{2}}(C(r))^{-\frac{1}{r}}.$
\end{proof}
\begin{corollary}\label{str_mixing_discrete}
Let $\rwD$ be a distribution {satisfying the conditions in \Cref{dyadic_coupling}, and recall that $f$ is a function satisfying \eqref{beginning}.} Take an arbitrary value for $f(0)$ so that it is defined on $\mathbb R^d$, then for any $\epsilon>0$ and $m>0$, there exists a constant $C(d,\rwD,\epsilon,m)>0$
\[
\rwP{\rwD}{0}\pars*{\absolute*{\frac{\sum_{i=0}^{n}f(\rwS{i})}{\log n}-\gnE \bracks*{\int_1^e f(B_t)dt}}>\epsilon}\le C(d,\rwD,\epsilon,m) (\log n )^{-m},\,\forall n\ge 0.
\]
\end{corollary}
\begin{proof}
Extend the discrete process $(\rwS{n})_{n\in\mathbb{N}}$ to a continuous-time process $(\rwS{\lfloor t\rfloor})_{t\ge0}.$ {Using \Cref{dyadic_coupling} and some basic estimates} on the Brownian motion, we can find a Brownian motion with the same covariance matrix as $\rwS{}$ on the same probability space, a constant $C(d,\rwD)>0$, and a power index $k$, such that the event 
\[
F_n:=\left\{\max_{0\le t\le n}\left|\rwS{t}-B_t\right|<C(d,\rwD)(\log n)^2\right\}\cap\left\{\inf\limits_{t>(\log n)^k}|B_t|>(\log n)^3\right\}
\]
happens with probability $1-O((\log{n})^{-m})$. 

{Recall that $f$ is continuous on $\mathbb R^d\backslash\braces{0}$ and homogeneous of degree $2$, we can easily get that for any $\delta>0$, when $n$ is large enough,
for any $x,y\in\mathbb R^d$ such that $|y|>(\log n)^3$, $|x-y|<(\log n)^2$,}
\begin{align*}
|f(x)-f(y)|
&\le \absolute*{|x|^{-2}-|y|^{-2}}f\pars*{\frac{x}{|x|}}
+|y|^{-2}\absolute*{f\pars*{\frac{x}{|x|}}-f\pars*{\frac{y}{|y|}}}\\
&= |y|^{-2}\frac{|x|+|y|}{|x|}\frac{|x|-|y|}{|x|}f\pars*{\frac{x}{|x|}}
+|y|^{-2}\absolute*{f\pars*{\frac{x}{|x|}}-f\pars*{\frac{y}{|y|}}}\\
&\le \delta |y|^{-2}.
\end{align*}
Therefore, conditioned on the event $F_n$, if we {write} $C_f=\gnE \bracks*{\int_1^e f(B_t)dt}$ for simplicity, we have that
\begin{align*}
&\left|\sum_{i=0}^{n-1} f(\rwS{i})-C_f\log n\right|\\
\le&\sum_{i=0}^{\lceil(\log n)^k\rceil} f(\rwS{i})
+\left|\int_{(\log n)^k}^{n} f(B_t)dt-C_f\log n\right|
+\int_{(\log n)^k}^{n} \left|f(B_t)-f(\rwS{t})\right|dt\\
\le&
\sum_{i=0}^{\lceil(\log n)^k\rceil} f(\rwS{i})
+\left|\int_{(\log n)^k}^{n} f(B_t)dt-C_f\log n\right|
+\delta\int_{(\log n)^k}^{n}|B_t|^{-2}dt.
\end{align*}
For the first term, by \Cref{A.2}, we have that
\[
\gnP\pars*{\sum_{i=0}^{\lceil(\log n)^k\rceil} f(\rwS{i})>\epsilon \log n}\le C_1(d,\rwD,\epsilon,m)(\log n)^{-m}.
\]
Similar bounds for the second and the third term follows from \Cref{str_mixing}.
\end{proof}
\section{The super-critical dimensions}\label{highDimSection}
In this section, we prove \Cref{mainresult} for $d\ge 7$ via the infinite model defined in \Cref{hd_model}. The main strategy  is to establish a lower bound for the expectation of capacity {using} \Cref{capA} and estimates on Green's functions, then {deduce the desired convergence for the infinite model with the help of its ergodicity under transformation \eqref{hdSh}, and finally extend it to a similar convergence for the branching random walk indexed by the critical Galton-Watson tree conditioned to be large.}

\subsection{Estimates on Green's functions}
\begin{lemma}\label{est_Green}
{If $d\ge 3$ and $\rwD,\brwD$ are distributions on $\mathbb Z^d$ satisfying \eqref{assumption},} then as $n\rightarrow\infty$, there exists a constant $C(d,\rwD,\brwD)>0$ such that
\[
\rwE{\brwD}{0}[\green{\rwD}(\rwS{n})]\le C(d,\rwD,\brwD)n^{1-d/2}.
\]
\end{lemma}{}
\begin{proof}
{Recall that according to \Cref{Green_asymptotic}, there exists $C'(d,\rwD,\brwD)$ such that
\[(C'(d,\rwD,\brwD))^{-1}\green{\brwD}(x) \le \green{\rwD}(x)\le  C'(d,\rwD,\brwD)\green{\brwD}(x) \quad \text{uniformly for all } x\in\mathbb Z^d,\]}
then it suffices to show that
\[
\rwE{\brwD}{0}[\green{\brwD}(\rwS{n})]\le C(d,\brwD)n^{1-d/2}.
\]
{In fact} since $\brwD$ is a symmetric distribution,
\begin{align*}
\rwE{\brwD}{0}[\green{\brwD}(\rwS{n})]
&=\sum_{x\in\mathbb Z^d}\rwP{\brwD}{0}(\rwS{n}=x)\green{\brwD}(x)\\
&=\sum_{x\in\mathbb Z^d}\sum_{m\ge 0}
\rwP{\brwD}{0}(\rwS{n}=x)
\rwP{\brwD}{0}(\rwS{m}=x)\\
&=\sum_{m\ge 0}
\rwP{\brwD}{0}(\rwS{m+n}=0)\le C(d,\brwD)n^{1-d/2},
\end{align*}
where the last line follows by taking $x=0$ in \Cref{LCLT} {and this completes the proof}. 
\end{proof}{}
\begin{lemma}\label{Green_sum}
In dimension $d\ge 7$, {recall that $\brwCh,\brwD,\rwD$ are probability distributions satisfying  \eqref{assumption}, and the sequence $(v_i)$ of the infinite model is defined in \eqref{range_tree}.} Then there exists a constant $C(d,\brwCh,\brwD,\rwD)>0$ such that
\[\hdE\bracks*{\sum_{i=-\infty}^\infty \green{\rwD}(\hdX{i})}\le C(d,\brwCh,\brwD,\rwD).\]
\end{lemma}

\begin{proof}
{Recall that $(v_i)$ run through all subtrees denoted by $\hdTi{\pm n}=\hdTi{n}\cup \hdTi{-n}$, thus}
\begin{align*}
&\hdE\bracks*{\sum_{i=-\infty}^\infty \green{\rwD}(\hdX{i})}\\
=&\hdE\otimes\rwE{\brwD}{0}
\bracks*{\sum_{n=0}^\infty\sum_{i=0}^\infty 
\#\setof{\gnNd\in\hdTi{\pm n}}{|\gnNd|=i}
\green{\rwD}(\rwS{n+i})}\\
=&
\sum_{n=0}^\infty\sum_{i=0}^\infty 
\hdE\bracks*{\#\setof{\gnNd\in\hdTi{\pm n}}{|\gnNd|=i}}
\rwE{\brwD}{0}\bracks*{\green{\rwD}(\rwS{n+i})}.
\end{align*}
If $\brwCh$ has finite variance, then for all $n$ and $i$ ,
\[
\hdE\bracks*{\#\setof{\gnNd\in\hdTi{\pm n}}{|\gnNd|=i}}=\hdE\bracks*{\#\setof{\gnNd\in\hdTi{\pm n}}{|\gnNd|=1}}= \sum_{i,j\ge0}(i+j)\brwCh(i+j+1),
\]
thus we have that
\begin{align*}
\hdE\bracks*{\sum_{i=-\infty}^\infty \green{\rwD}(\hdX{i})}
&\le C(\brwCh)\sum_{n=0}^\infty\sum_{i=0}^\infty 
\rwE{\brwD}{0}\bracks*{\green{\rwD}(\rwS{n+i})}\\
&=C(\brwCh)\sum_{m=0}^\infty (m+1)\rwE{\brwD}{0}\bracks*{\green{\rwD}(\rwS{m})}
\le C(\brwCh,d,\rwD,\brwD),
\end{align*}
where the last line follows from \Cref{est_Green}.
\end{proof}

\subsection{Limit theorem for the infinite model}
\begin{proposition}\label{hdconclusion_infinite}
In dimension $d\ge 7$, {$\brwCh,\brwD,\rwD$ are supposed to satisfy \eqref{assumption} and recall the range $\hdR{0}{n}$ defined in \Cref{hd_model}. }Then
there is a constant $C(d,\brwCh,\brwD,\rwD)>0$ such that
\[
\frac{\capa\pars*{\hdR{0}{n}}}{n}\rightarrow C(d,\brwCh,\brwD,\rwD) \quad \hdP\text{-almost surely}.
\]
\end{proposition}
\begin{proof}
By definition of the capacity, for any finite sets $A,B\subset\mathbb Z^d$,
\[
\capa(A\cup B)\le\capa A+\capa B.
\]
Recall the ergodic measure-preserving shift $\hdSh$ {defined by \eqref{hdSh}. In particular} we have that
\begin{align*}
\capa\pars*{\hdR{0}{n+m}}&\le 
\capa\pars*{\hdR{0}{n}}+\capa\pars*{\hdR{n}{n+m}}\\
&=\capa\pars*{\hdR{0}{n}}+\capa\pars*{\hdSh^{n}\circ\hdR{0}{m}}.
\end{align*}
{Thus Kingman's subadditive ergodic theorem suggests that} there exists a constant $C(d,\brwCh,\brwD,\rwD)$ such that
\[
\lim_{n\rightarrow\infty}\frac{\capa\pars*{\hdR{0}{n}}}{n}\rightarrow C(d,\brwCh,\brwD,\rwD) \quad \hdP\text{-almost surely}.
\]
Then it remains to prove that the constant 
\begin{equation}\label{eq:hd_inf}
C(d,\brwCh,\brwD,\rwD) = \lim_{n\rightarrow\infty} \frac{1}{n}\hdE\bracks*{\capa\pars*{\hdR{0}{n}}}
\end{equation}
is strictly positive. 

In face by \Cref{capA}, for any $k\ge 1$,
\[
\frac{1}{n}\hdE\bracks{\capa\pars*{\hdR{0}{n}}}
\ge 
\frac{\frac{1}{n}\hdE\bracks*{\#\hdR{0}{n}}}{k+1}-\frac{\frac{1}{n}\hdE\bracks*{\sum_{x,y\in \hdR{0}{n}}\green{\rwD}(x,y)}}{k(k+1)}.
\]
The first term $\frac{1}{n}\hdE\bracks*{\#\hdR{0}{n}}$ converges to a strictly positive constant by~\cite[Proposition 5]{LeGall-Lin-range}, and in the second term 
\begin{align*}
\frac{1}{n}\hdE\bracks*{\sum_{x,y\in \hdR{0}{n}}\green{\rwD}(x,y)}
\le\frac{1}{n}\hdE\bracks*{\sum_{i,j=0}^n\green{\rwD}(\hdX{i},\hdX{j})}
\le\hdE\bracks*{\sum_{i=-\infty}^\infty\green{\rwD}(\hdX{i})}
\end{align*}
is also finite by \Cref{Green_sum}. {Then \eqref{eq:hd_inf} is strictly positive by taking $k$ sufficiently large.}
\end{proof}
\begin{remark}\label{hd_constant}
The limiting constant here is implicit. In fact, in the language of \Cref{cap_relation}, for high dimensions $d\ge 7$, both $E_nI_n$ and $G_n$ will converge by monotonicity (to some random variables). {Indeed, write $E_\infty I_\infty$ and $G_\infty$ to denote their limits, then the desired constant is}
\[
\gnE\bracks*{E_\infty I_\infty}=\hdP\otimes\rwP{0}{\rwD}\pars*{\hdX{0}\not\in\braces{\hdX{1},\hdX{2},\dots},\tau^+_{\braces{\dots,\hdX{-1},\hdX{0},\hdX{1},\dots}}=\infty}
.\]
However, the equation $\gnE\bracks*{E_\infty I_\infty\cdot G_\infty}=1$ does not contain enough information to determine this constant, since $G_\infty$ is a non-trivial random variable for $d\ge 7$.
\end{remark}
\subsection{Proof of \texorpdfstring{\Cref{mainresult}}{} (1)}
The goal of this section is to establish an intermediate structure, then compare the infinite model with large Galton-Watson trees via this new structure as in \cite[p. 19]{zhu-cbrw}. 

To study $\hdR{0}{n}$, it suffices to look at $(\hdX{i})$ for $i\ge 0$, thus we consider the model in \Cref{half_model}, {i.e. we attach} one subtree $\hdTi{i}$ to each node $(i,\varnothing)$ on the spine and set 
\[\gnCh{(0,\varnothing)}{\sim}\brwCh,\,\hdP(\gnCh{(i,\varnothing)}=n)=\brwCh[n+1,\infty)=\sum_{j=n+1}^\infty\brwCh(j),\,i>0.\]
Now we {construct} a new probability measure $\hdPplus$ such that all nodes on the spine, including the base point $(0,\varnothing)$, have offspring distribution
\[\hdPplus(\gnCh{(i,\varnothing)}=n)=\brwCh[n+1,\infty),\,i\ge 0,\]
while all other constructions {(independence, offspring distribution for nodes not on the spine, and displacements)} are the same as $\hdP$. Since $\hdPplus$ and $\hdP$ are different only in the first subtree, it follows that

\begin{corollary}\label{new_hd_1}
In dimension $d\ge 7$, let $\brwCh,\brwD,\rwD$ be distributions with the conditions in \eqref{assumption}.
There is a constant $C(d,\brwCh,\brwD,\rwD)>0$ such that under $\hdPplus$,
\[
\frac{\capa\pars*{\hdR{0}{n}}}{n}\rightarrow C(d,\brwCh,\brwD,\rwD) \text{ in probability}.
\]
\end{corollary}
Moreover, for the measure $\hdPplus$ we have 
\begin{lemma}[\cite{zhu-cbrw}]\label{new_hd_2}
In dimension $d\ge 3$, let $\brwCh,\brwD,\rwD$ be distributions with the conditions in \eqref{assumption}. Recall that $\brwP$ is the law of {the} Galton-Watson tree (cf. \Cref{GWdef}). Let $a\in(0,1)$ and let $(f_n)$ be any uniformly bounded {sequence} of functions on $\mathbb Z^{\lfloor an\rfloor+1}$. Then (with an abuse of the {notation} $(\hdX{i})$ for positions of nodes under {both} $\brwP$ and $\hdPplus$)
\[
\lim_{n\rightarrow\infty}\absolute*{
\brwE\parsof*{f_n\pars*{(\hdX{i})_{0\le i\le\lfloor an\rfloor}}}{\#T=n}
-\hdEplus\pars*{f_n\pars*{(\hdX{i})_{0\le i\le\lfloor an\rfloor}}
g_a\pars*{\frac{L_{\lfloor an\rfloor}}{\sigma n}}}
}=0,
\]
where $g_a(x)=(1-a)^{-\frac{3}{2}}\exp\pars*{-\frac{x^2}{2(1-a)}}$, $\sigma^2$ is the variance of $\brwCh$, and $(L_i)$ is the corresponding Lukasiewisz path defined by (recall that $k_u$ denotes the number of children of $u$)
\[
L_0=0,\,L_{i+1}-L_i=k_{u_i}-1.
\]
\end{lemma}
{\begin{proof}
See (5.3), (5.4) and the {disaplay} that follows in \cite{zhu-cbrw}.
\end{proof}}
\begin{theorem}\label{HighDimConclusion}
In dimension $d\ge 7$, let $\brwCh,\brwD,\rwD$ be distributions with the conditions in \eqref{assumption}, and let $\hdR{0}{n}$ be the range constructed in \Cref{hd_model} (abused to denote the range of other trees as well).
There is a constant $C=C(d,\brwCh,\brwD,\rwD)>0$ such that
under the law of a (standard) Galton-Watson tree conditioned to have $n+1$ nodes, $\brwP(\cdot|\#T=n+1)$,
\[
\frac{\capa(\hdR{0}{n})}{n}\rightarrow C\text{ in probability.}
\]
\end{theorem}
\begin{proof}
For any $\epsilon>0$, take 
\[
f_n=\mathbf 1_{\absolute*{\frac{1}{n}\capa\hdR{0}{an}-aC}>\epsilon}
\]
in \Cref{new_hd_2}. Then by \Cref{new_hd_1}, we have that
\begin{equation}\label{eq:explain37}
\lim_{n\rightarrow\infty}
\brwP\parsof*{\absolute*{\frac{1}{n}\capa\hdR{0}{an}-aC}>\epsilon}{\#T=n+1}=0,
\end{equation}

Moreover, since 
\begin{align*}
   &\absolute*{\frac{1}{n}\capa(\hdR{0}{n})-C}\\
\le&\absolute*{
\frac{1}{n}\capa(\hdR{0}{n})-\frac{1}{n}\capa(\hdR{0}{\lfloor an \rfloor})}+ \absolute*{\frac{1}{n}\capa(\hdR{0}{\lfloor an \rfloor})-aC}+|aC-C|\\
\le&(1-a)+\absolute*{\frac{1}{n}\capa(\hdR{0}{\lfloor an \rfloor})-aC}+(1-a)C,
\end{align*} 
we have that
\begin{align*}
   &\lim_{n\rightarrow\infty}
   \brwP\parsof*{\absolute*{\frac{1}{n}\capa(\hdR{0}{n})-C}>\epsilon}{\#T=n+1}\\
\le&\lim_{n\rightarrow\infty}\brwP
   \parsof*{\absolute*{\frac{1}{n}\capa(\hdR{0}{\lfloor an \rfloor})-aC}>\epsilon-(1-a)(1+C)}{\#T=n+1} =0,
\end{align*}
{where the last line holds {by \eqref{eq:explain37}} if $a$ is taken sufficiently close to $1$.}
\end{proof}

\section{The critical dimension}\label{TheCriticalDimension}

In this section, we consider the critical dimension $d=6$. The main strategy is to estimate Green's functions for the infinite model established in \Cref{hd_model}, so that we can use \Cref{cap_relation} and a second moment method to get the desired convergence. Finally similar argument as in  \Cref{HighDimConclusion} allows us to prove the convergence result of capacity for large Galton-Watson trees. 

\subsection{Estimates on Green's functions}
\begin{proposition}\label{moments_Green} 
In dimension $d=6$, let $\brwCh,\brwD,\rwD$ be distributions with assumptions in \eqref{assumption}. Let $\brwP$ be the law of the (standard) branching random walk $(\gnX{\gnNd})_{\gnNd\in\gnT}$ indexed by the (standard) Galton-Watson tree $\gnT$ (cf. \Cref{GWdef}). Then
\begin{enumerate}
\item 
As $z\rightarrow\infty$, we have that
\[\brwE\bracks*{\sum_{\gnNd\in\gnT}\green{\rwD}(z+\gnX{\gnNd})}=F_{\rwD,\brwD}(z)+O(|z|^{-3}),\]
where the function 
\begin{equation}\label{eq:1.1}
F_{\rwD,\brwD}(z):=C_{6,\rwD}C_{6,\brwD}\int_{\mathbb R^6}\rwJ{\rwD}(z+x)^{-4}\rwJ{\brwD}(x)^{-4}dx,
\end{equation}
is a continuous function defined on $\mathbb R^6\backslash\{0\}$
with $F_{\rwD,\brwD}(\lambda z)=\lambda^{-2}F(z)$ for all $\lambda>0,$ with $C_{6,(\cdot)}$ and $\rwJ{(\cdot)}$ defined in \Cref{Green_asymptotic}.
\item For any $m\ge 2$, if $\brwCh$ has finite $m$-th moment, then there exists a constant $C(m,\brwCh,\brwD,\rwD)>0$, so that for any $z\ne 0$,
\[\brwE\bracks*{\pars*{\sum_{\gnNd\in\gnT}\green{\rwD}(z+\gnX{\gnNd})}^m}\le C(m,\brwCh,\brwD,\rwD)|z|^{-2}.\]
\end{enumerate}{}
\end{proposition}
\begin{proof}
Because $\brwCh$ is critical, we have $\brwE \bracks*{\#\setof*{\gnNd\in\gnT}{|u|=n}}=1$ for all $n\ge 1$. 
Then 
\begin{align*}
\brwE\bracks*{\sum_{\gnNd\in\gnT}\green{\rwD}(z+\gnX{\gnNd})}
&=\brwE\bracks*{
\sum_{n=0}^\infty 
\#\setof*{\gnNd\in\gnT}{|u|=n}
\rwE{\brwD}{0}[\green{\rwD}(z+\rwS{n})]}\\
&=\sum_{n=0}^\infty
\rwE{\brwD}{0}[\green{\rwD}(z+\rwS{n})]\\
&=\sum_{n=0}^\infty\sum_{x\in\mathbb Z^6}
\green{\rwD}(z+x)\rwP{\brwD}{0}(\rwS{n}=x)\\
&=\sum_{x\in\mathbb Z^6}\green{\rwD}(z+x)\green{\brwD}(x).
\end{align*}

By \Cref{Green_asymptotic}, we then have
\begin{align*}
 &\sum_{x\in\mathbb Z^6}\green{\rwD}(z+x)\green{\brwD}(x)\\
=&C_{6,\rwD}C_{6,\brwD}\sum_{x\in\mathbb Z^6}\rwJ{\rwD}(z+x)^{-4}\rwJ{\brwD}(x)^{-4}+
O\pars*{\sum_{x\in\mathbb Z^6}|z+x|^{-5}|x|^{-4}},
\end{align*}
and it is elementary to show that (by approximating the sum by an integral)
\begin{align*}
O\pars*{\sum_{x\in\mathbb Z^6}|z+x|^{-5}|x|^{-4}}
=O(|z|^{-3}).
\end{align*}
Moreover, the difference between $C_{6,\rwD}C_{6,\brwD}\sum_{x\in\mathbb Z^6}\rwJ{\rwD}(z+x)^{-4}\rwJ{\brwD}(x)^{-4}$ and $F_{\rwD,\brwD}(z)$ is of the same order as $O\pars*{\sum_{x\in\mathbb Z^6}|z+x|^{-5}|x|^{-4}}$ by the mean value theorem. Therefore,
\[\brwE\bracks*{\sum_{\gnNd\in\gnT}\green{\rwD}(z+\gnX{\gnNd})}=F_{\rwD,\brwD}(z)+O(|z|^{-3}).\]
The asymptotic and the scaling relation for $F_{\rwD,\brwD}$ are easy to check by using $\rwJ{\cdot}(x)\asymp |x|,\rwJ{\cdot}(\lambda x)=\lambda\rwJ{\cdot}(x)$.

As for Part (2), let $(\rwS{n}^{(i)})(1\le i\le k)$ be independent $\brwD$-random walks started at 0. Given any $z\in\mathbb Z^6,k\ge 2$, by Part (1) and \Cref{Green_asymptotic},
\begin{equation}\label{hierarchy2}
\begin{aligned}
\rwE{\brwD}{0}\bracks*{\prod_{i=1}^k\sum_{j=0}^\infty \green{\rwD}(z+\rwS{j}^{(i)})}
\le C_1(\brwD,\rwD)\prod_{i=1}^k(|z|\vee 1)^{-2}\le C_2(\brwD,\rwD)\green{\rwD}(z)^{k/2}.
\end{aligned}
\end{equation}

To deal with the second moment, $m=2$, we need to study the positions of two nodes $\gnNd,\gnNd'$. Given that 
$|\gnNd\wedge\gnNd'|=k,|\gnNd|=k+i,|\gnNd'|=k+j$, where $\gnNd\wedge\gnNd'$ denotes their youngest common ancestor), then their contribution to the second moment is
\[
\rwE{\brwD}{0}\green{\rwD}(z+\rwS{k}+\rwS{i}^{(1)})\green{\rwD}(z+\rwS{k}+\rwS{j}^{(2)}).
\]
Summing up all possible tree-structures, we have that
\begin{align*}
&\brwE\bracks*{\pars*{\sum_{\gnNd\in\gnT}\green{\rwD}(z+\gnX{\gnNd})}^2}\\
=&
\sum_{i,j,k=0}^\infty\rwE{\brwD}{0}\bracks*{ \green{\rwD}(z+\rwS{k}+\rwS{i}^{(1)})\green{\rwD}(z+\rwS{k}+\rwS{j}^{(2)})}
\brwE[N(k;i,j)],
\end{align*}
where
\[
N(k;i,j)=\#\setof*{\gnNd,\gnNd'\in\gnT}{|\gnNd\wedge\gnNd'|=k,|\gnNd|=k+i,|\gnNd'|=k+j}.
\]
\begin{figure}[ht]
\centering
\includegraphics{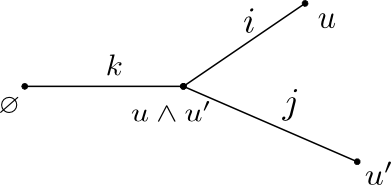}
\caption{$N(k;i,j)$}
\label{figure411}
\end{figure}

We can then count $N(k;i,j)$ as illustrated in \Cref{figure411} on critical Galton-Watson trees. Set $Z_n:=\#\setof{\gnNd\in\gnT}{|u|=n}$), then
\[
\brwE[N(k;i,j)]=\brwE [Z_k]\brwE[Z_1(Z_1-1)]\brwE [Z_{i-1}]\brwE [Z_{j-1}]=\brwE[Z_1(Z_1-1)]
\]
for $i,j\ge 1$, which is finite as long as $\brwCh$ has finite second moment (the case $i$ or $j=0$ can be easily treated alone). Then we apply \eqref{hierarchy2} with $k=2$,
\begin{align*}
&\brwE\bracks*{\pars*{\sum_{\gnNd\in\gnT}\green{\rwD}(z+\gnX{\gnNd})}^2}\\
\le& \brwE[Z_1(Z_1-1)]
\sum_{i,j,k=0}^\infty\rwE{\brwD}{0}\bracks*{ \green{\rwD}(z+\rwS{k}+\rwS{i}^{(1)})\green{\rwD}(z+\rwS{k}+\rwS{j}^{(2)})}\\
\le& C(\brwD,\rwD)\brwE[Z_1(Z_1-1)]
\sum_{k=0}^\infty\rwE{\brwD}{0}\bracks*{ \green{\rwD}(z+\rwS{k})}
\le C(\brwCh,\brwD,\rwD)|z|^{-2},
\end{align*}
where the last inequality follows from Part (1).

Similar argument works for $m\ge 3$, by counting all possible hierarchy structures of $m$ vertices as for $N(k;i,j)$, and perform \eqref{hierarchy2} recursively on those structures.
\end{proof}

\begin{remark}\label{remark412}
By \eqref{hierarchy2}, one may expect an $O(|z|^{-m})$ result in Part (2), however, $O(|z|^{-2})$ is in fact optimal for all $m\ge 3$. Take $m=3$ for instance. To estimate the contribution of 'binary' branching structure $\gnNd^{(i)} (i=1,2,3)$ with (see \Cref{figure412})
\begin{figure}[ht]
\centering
\includegraphics{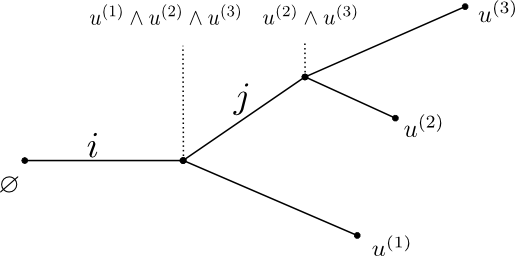}
\caption{'binary' branching structures for $k=3$}
\label{figure412}
\end{figure}
\[
|\gnNd^{(1)}\wedge \gnNd^{(2)}\wedge \gnNd^{(3)}|=i,|\gnNd^{(2)}\wedge\gnNd^{(3)}|=i+j>i,
\]
we need to perform \eqref{hierarchy2} with $k=2$ twice, instead of the equation with $k=3$:
\begin{align*}
&\sum_{i,j,k,l,h=0}^\infty\rwE{\brwD}{0}\bracks*{ \green{\rwD}(z+\rwS{i}+\rwS{j}^{(1)}+\rwS{k}^{(2)})\green{\rwD}(z+\rwS{i}+\rwS{j}^{(1)}+\rwS{l}^{(3)})\green{\rwD}(z+\rwS{i}+\rwS{h}^{(4)})}\\
\le 
&C_1(\brwCh,\brwD,\rwD)\sum_{i,j,h=0}^\infty\rwE{\brwD}{0}\bracks*{ \green{\rwD}(z+\rwS{i}+\rwS{j}^{(1)})\green{\rwD}(z+\rwS{i}+\rwS{h}^{(4)})}\\\le
&C_2(\brwCh,\brwD,\rwD)\sum_{i=0}^\infty\rwE{\brwD}{0}\bracks*{ \green{\rwD}(z+\rwS{i})}\le C_3(\brwCh,\brwD,\rwD)|z|^{-2}.
\end{align*}
It is only when $u^{(1)},u^{(2)},u^{(3)}$ all branch at the same node (i.e. $j=0$ in \Cref{figure412}) that one can apply \eqref{hierarchy2} with $k=3$.
Thus our method gives the bound $O(|z|^{-2})$ for all $m$-th moment for $m\ge2$.
\end{remark}
Since the infinite model has offspring distributions different from $\brwCh$ only for nodes on the spine, we include the following corollary, whose proof is clear by that of \Cref{moments_Green}.
\begin{corollary}\label{cor_subtree}
In the setting of \Cref{moments_Green}, take an arbitrary distribution ${\brwCh}^*$ on $\mathbb N$, we consider the random tree whose offspring distribution differs from that of $\brwP$ only in the first generation, replaced by ${\brwCh}^*$. The resulted distribution on branching random walks indexed by the modified random tree is denoted by ${\brwPst}$. Then
\begin{enumerate}
\item 
As $z\rightarrow\infty$, we have that
\[{\brwEst}\bracks*{\sum_{\gnNd\in\gnT}\green{\rwD}(z+\gnX{\gnNd})}=\gnE[{\brwCh}^*]F_{\rwD,\brwD}(z)+O(|z|^{-3}).\]
\item For any $m\ge 2$, if ${\brwCh}^*$ and $\brwCh$ have finite $m$-th moment, then there exists a constant $C(m,\brwCh,\brwCh^*,\brwD,\rwD)>0$
\[{\brwEst}\bracks*{\pars*{\sum_{\gnNd\in\gnT}\green{\rwD}(z+\gnX{\gnNd})}^m}\le C(m,\brwCh,\brwCh^*,\brwD,\rwD)|z|^{-2}.\]
\end{enumerate}{}
\end{corollary}

Before going to the main estimate, we attach here a moment estimate for independent random variables.
\begin{lemma}\cite[Corollary 4.4]{fn1971probability}\label{lm:large_deviation}
    Let $m\ge 2$, and $(X_i),i=1,\dots,n$ be independent random variables such that
    \[\mathbb E X_i=0,\text{ and }\mathbb E|X_i|^m<\infty,\]
    then
    \[\mathbb P\left(\sum_{i=1}^nX_i\ge x\right)\le C_1 x^{-m}\sum_{i=1}^n \mathbb E|X_i|^m+\exp\left(-{C_2x^2}/{\sum_{i=1}^n \mathbb E|X_i|^2}\right),\]
    where $C_1=(1+2/m)^m,C_2=2(m+2)^{-1}e^{-m}.$
\end{lemma}
We are now ready to treat Green's functions for the infinite model. 
\begin{proposition}\label{n-tree}
In dimension $d=6$, let $\brwCh,\brwD,\rwD$ be distributions with assumptions in \eqref{assumption}. Recall the infinite model in \Cref{hd_model}. Let $\zeta_{-n},\zeta_{n}$ be indexes such that
\[
\hdR{\zeta_{-n}}{\zeta_{n}}=\{v_{\zeta_{-n}},\dots,v_{\zeta_n}\}
\]
is the range formed by the displacement of all nodes in 
\[
\braces*{(0,\hdTi{0}),(1,\hdTi{\pm 1}),\dots,(n,\hdTi{\pm n})}.
\]
\begin{enumerate}
\item 
If $\brwCh$ has finite $5$-th moment, then
for any fixed $\epsilon>0$, as $n\rightarrow\infty$,
\[
\hdP\pars*{\absolute*{\sum_{i=\zeta_{-n}}^{\zeta_n}\green{\rwD}(\hdX{i})-\Cfconst\log n}>\epsilon\log n}=o((\log n)^{-2})
\]
where 
\begin{equation}\label{eq:1.3}
\Cfconst=\sum_{k=1}^\infty (k-1)k\brwCh(k)\cdot \gnE \bracks*{\int_1^e F_{\rwD,\brwD}(B_t^\brwD)dt},
\end{equation}
$B_t^\brwD$ is the Brownian motion with covariance matrix $\rwVar{\brwD}$, and $F_{\rwD,\brwD}$ is the function defined in \eqref{eq:1.1}.
\item
For any $m\ge 2$, if $\brwCh$ has finite $(m+1)$-th moment, then as $n\rightarrow\infty$,
\[
\hdE\bracks*{\left(\sum_{i=\zeta_{-n}}^{\zeta_n}\green{\rwD}(\hdX{i})\right)^m}=O((\log n)^{m}).
\]
\end{enumerate}
\end{proposition}
\begin{proof}
We merge the two subtrees $(n,\hdTi{\pm n})$ into a single tree, whose first generation has offspring distribution 
\begin{equation}\label{eq:1.4}
{\brwCh}^*(k):=\sum_{\setof{i,j}{i+j=k}}\brwCh(i+j+1)=(k+1)\brwCh(k+1).
\end{equation}
Then we need $\brwCh$ to have finite $(m+1)$-th moment in order that ${\brwCh}^*$ has finite $m$-th moment. For simplicity, we denote by $\green{\rwD}(\hdTi{\pm n})$ the sum of Green's functions over the range of $(n,\hdTi{\pm n})$, and we denote by $\spine{0}=0,\spine{1},\dots,\spine{n}$ the spatial positions of the spine $(0,\varnothing),\dots,(n,\varnothing)$. Clearly,
\[
\sum_{i=\zeta_{-n}}^{\zeta_n}\green{\rwD}(\hdX{i})=
\green{\rwD}(\hdTi{0})+\sum_{i=1}^n\green{\rwD}(\hdTi{\pm i}),
\]
and $(\green{\rwD}(\hdTi{\pm i}))$ are independent conditioned on $(\spine{i})$.

For Part (1),
we have that

\begin{equation}\label{absolute_divide}
\begin{aligned}
\absolute*{\sum_{i=\zeta_{-n}}^{\zeta_n}\green{\rwD}(\hdX{i})-\Cfconst\log n}\le
&
\absolute*{\sum_{i=0}^n\green{\rwD}(\hdTi{\pm i})-\sum_{i=0}^n\hdE\bracksof*{\green{\rwD}(\hdTi{\pm i})}{(\spine{i})_{0\le i\le n}}}\\
&+\absolute*{\sum_{i=0}^n\hdE\bracksof*{\green{\rwD}(\hdTi{\pm i})}{(\spine{i})_{0\le i\le n}}-\gnE[{\brwCh}^*]\sum_{i=1}^n F_{\rwD,\brwD}(\spine{i})}\\
&+
\absolute*{\gnE[{\brwCh}^*]\sum_{i=1}^n F_{\rwD,\brwD}(\spine{i})-\Cfconst\log n}
\end{aligned}
\end{equation}
and it suffices to estimate each of the three terms here.

Indeed, for the third term in \eqref{absolute_divide}, by \Cref{str_mixing_discrete},
\[
\hdP\pars*{\absolute*{\gnE[{\brwCh}^*]\sum_{i=1}^n F_{\rwD,\brwD}(\spine{i})-\Cfconst\log n}>\epsilon\log n}=o((\log n)^{-2}),
\]
For the second term in \eqref{absolute_divide}, by \Cref{cor_subtree} we have that
\[
\absolute*{\sum_{i=0}^n\hdE\bracksof*{\green{\rwD}(\hdTi{\pm i})}{(\spine{i})_{0\le i\le n}}-\gnE[{\brwCh}^*]\sum_{i=1}^n F_{\rwD,\brwD}(\spine{i})}=O\pars*{\sum_{i=0}^n(|\spine{i}|\vee 1)^{-3}},
\]
which is in turn deduced by \Cref{A.2} (2) with $k=3,\,m=1,2,3$ (and a Chebyshev-type inequality for the $3$rd moment),
\[
\hdP\pars*{\sum_{i=0}^n(|\spine{i}|\vee 1)^{-3}>\epsilon\log n}=o((\log n)^{-2}).
\]
As for the first term in \eqref{absolute_divide}, by \Cref{cor_subtree} with $m=2,4$ (here we need finite fourth moment for ${\brwCh}^*$, thus finite fifth moment for $\brwCh$), we have that
\begin{equation*}
\begin{aligned}
\sum_{i=0}^n\hdE\bracksof*{(\green{\rwD}(\hdTi{\pm i}))^2}{(\spine{i})_{0\le i\le n}}&\le C_1(\brwCh,\brwD,\rwD)\sum_{i=0}^n(|\spine{i}|\vee1)^{-2},\\
\sum_{i=0}^n\hdE\bracksof*{(\green{\rwD}(\hdTi{\pm i}))^4}{(\spine{i})_{0\le i\le n}}&\le C_2(\brwCh,\brwD,\rwD)\sum_{i=0}^n(|\spine{i}|\vee1)^{-2},
\end{aligned}
\end{equation*}
therefore
\begin{equation*}
\begin{aligned}
\sum_{i=0}^n\hdE\bracksof*{(\green{\rwD}(\hdTi{\pm i})-\hdE\bracksof*{\green{\rwD}(\hdTi{\pm i})}{(\spine{i})_{0\le i\le n}})^2}{(\spine{i})_{0\le i\le n}}&\le C_3(\brwCh,\brwD,\rwD)\sum_{i=0}^n(|\spine{i}|\vee1)^{-2},\\
\sum_{i=0}^n\hdE\bracksof*{(\green{\rwD}(\hdTi{\pm i})-\hdE\bracksof*{\green{\rwD}(\hdTi{\pm i})}{(\spine{i})_{0\le i\le n}})^4}{(\spine{i})_{0\le i\le n}}&\le C_4(\brwCh,\brwD,\rwD)\sum_{i=0}^n(|\spine{i}|\vee1)^{-2}.
\end{aligned}
\end{equation*}
Then we apply \Cref{lm:large_deviation} with $X_i={\green{\rwD}(\hdTi{\pm i})-\hdE\bracksof*{\green{\rwD}(\hdTi{\pm i})}{(\spine{i})_{0\le i\le n}}}$, $m=4$, and $\mathbb P=\hdP\parsof{\cdot}{(\spine{i})_{0\le i\le n}}$,
\begin{equation*}
\begin{aligned}
&\hdP\parsof*{
\absolute*{
\sum_{i=0}^n \green{\rwD}({\hdTi{\pm i}})-
\hdE\bracksof*{\green{\rwD}(\hdTi{\pm i})}{(\spine{i})_{0\le i\le n}}
}
\ge\epsilon\log n}
{(\spine{i})_{0\le i\le n}}\\
\le &C_5(\brwCh,\brwD,\rwD)\sum_{i=0}^n(|\spine{i}|\vee1)^{-2}(\epsilon\log n)^{-4}+\exp\pars*{-C_6(\brwCh,\brwD,\rwD)(\epsilon\log n)^2/\sum_{i=0}^n(|\spine{i}|\vee 1)^{-2}}\\
\le &C_5(\brwCh,\brwD,\rwD)\sum_{i=0}^n(|\spine{i}|\vee1)^{-2}(\epsilon\log n)^{-4}+e^{-C_6(\brwCh,\brwD,\rwD)\epsilon^2\log n/\log\log n}+\mathbf 1_{\sum_{i=0}^n(|\spine{i}|\vee 1)^{-2}>\log n\log\log n},
\end{aligned}
\end{equation*}
If we take expectation $\hdE$ on both sides, all these terms are $o((\log n)^{-2})$ by \Cref{A.2}, then we have that
\[\hdP\pars*{
\absolute*{
\sum_{i=0}^n \green{\rwD}({\hdTi{\pm i}})-
\hdE\bracksof*{\green{\rwD}(\hdTi{\pm i})}{(\spine{i})_{0\le i\le n}}
}
\ge\epsilon\log n}=o((\log n)^{-2}).\]

The conclusion follows by combining the estimates for the three terms on the right-hand side of \eqref{absolute_divide} individually. 

For the second part, we illustrate the $m=2$ case, since the proof for $m$ other than $2$ is similar. Indeed,
\begin{align*}
	&\hdE\bracksof*{\left(\sum_{i=\zeta_{-n}}^{\zeta_n}
	\green{\rwD}(\hdX{i})\right)^2}{(\spine{i})_{0\le i\le n}}\\
	=&\hdE\bracksof*{\left(\sum_{i=0}^{n}\green{\rwD}(\hdTi{\pm i})\right)^2}
	{(\spine{i})_{0\le i\le n}}\\
	=&\sum_{i=0}^{n}\hdE\bracksof*{\green{\rwD}(\hdTi{\pm i})^2}{(\spine{i})_{0\le i\le n}}+\\
	&2\sum_{0\le i<j\le n}
	\hdE\bracksof*{\green{\rwD}(\hdTi{\pm i})}{(\spine{i})_{0\le i\le n}}
	\hdE\bracksof*{\green{\rwD}(\hdTi{\pm j})}{(\spine{i})_{0\le i\le n}}.
\end{align*}
By \Cref{cor_subtree}, if $\brwCh$ has finite $3$rd moment, then this sum is of the same order as
\begin{align*}
	&\sum_{i=0}^{n}(|\spine{i}|\vee 1)^{-2}+2\sum_{0\le i<j\le n}(|\spine{i}|\vee 1)^{-2}(|\spine{j}|\vee 1)^{-2}\\
	\asymp&\sum_{i=0}^{n}(|\spine{i}|\vee 1)^{-2}+\pars*{\sum_{i=0}^n(|\spine{i}|\vee 1)^{-2}}^2,
	\end{align*}
	Take expectation with respect to $\hdE$, and we can conclude by \Cref{A.2}.
\end{proof}

\begin{corollary}\label{core} 
Under the same setting of \Cref{n-tree} (1),
\[
\hdP\pars*{\absolute*{\sum_{i=-n}^{n}\green{\rwD}(\hdX{i})-\frac{1}{2}\Cfconst\log n}>\epsilon\log n}=o((\log n)^{-2}).
\]
\end{corollary}
\begin{proof}
	By standard tools of Kemperman's formula (see e.g. \cite[Section 3]{total_progeny}), denote by $\zeta'_n$ the total population of $n$ Galton-Watson trees of offspring distribution $\brwCh$, and by $(Y_i)$ an i.i.d. sequence distributed as $\brwCh-1$, then 
	\begin{align*}
	\hdP(\zeta'_n=m)=\frac{n}{m}\mathbb P(Y_1+\dots+Y_m=n).
	\end{align*}
	Apply \Cref{LCLT} with $d=1$ and the random walk with displacements $(Y_i)$ (where $\brwCh$ being critical implies that $\mathbb E Y_i=0$, and finite fifth moment required in \Cref{n-tree} (1) implies the finite third moment of $Y_i$), we have that
	\begin{align*}
	\absolute*{\hdP(\zeta'_n=m)-\frac{n}{m}{
	\frac{C_1(\brwCh)}{\sqrt m}e^{-\frac{C_2(\brwCh)n^2}{m}}}}\le\frac{C_3(\brwCh)}{m}.
	\end{align*}
    Sum over $m$, then
	\[
	\hdP\left(\zeta'_n\ge n^2(\log n)^5\right)=o((\log n)^{-2}).
	\]
	Moreover, by \cite[Proposition 2.1.2 (a)]{Lawler-book-RW} with $k=2$ (guaranteed by the finite fifth moment in \Cref{n-tree} (1)),
	\begin{align*}
	&\hdP\left(\zeta'_n\le n^2(\log n)^{-2}\right)\\
	\le&\sum_{m=1}^{n^2(\log n)^{-2}}\frac{n}{m}\mathbb P(Y_1+\dots+Y_m=n)\\
	\le&\pars*{\sum_{m=1}^{n^2(\log n)^{-2}}\frac{n}{m}}\cdot\gnP\pars*{\max_{1\le j\le n^2(\log n)^{-2}}Y_1+\dots+Y_i\ge n}\\
	=&o((\log n)^{-2}).
	\end{align*}
	In summary, 
	\[
	\hdP\left(n^2(\log n)^{-2}<\zeta'_n<n^2(\log n)^5\right)=1-o((\log n)^{-2}).
	\]
	
	Moreover, recall the probability distribution $\brwCh^*$ in \eqref{eq:1.4}. If we take an i.i.d. sequence $(X_i)$ distributed as ${{\brwCh}^*}$, then
	\[\zeta_n\overset{d}{=}\zeta'_{1+X_1+\dots+X_n}.\]
	Apply \cite[Proposition 2.1.2 (a)]{Lawler-book-RW} again for the sequence $(X_i-\mathbb EX_i)$, we can show that for any constants $0<C_4(\brwCh)<\mathbb EX_i<C_5(\brwCh)$,
	\[
	\mathbb P(C_4(\mu)n<1+X_1+\dots+X_n<C_5(\mu)n)=1-o((\log n)^{-2}).
	\]
	Thus for any $0<C_6(\brwCh)<(\mathbb E[X_i])^2<C_7(\brwCh)$,
	\begin{equation}\label{size_tree}
	\hdP\left(C_6(\mu)n^2(\log n)^{-2}<\zeta_n<C_7(\mu)n^2(\log n)^5\right)=1-o((\log n)^{-2}).
	\end{equation}
	The same estimate holds for $\zeta_{-n}$, thus we conclude by \Cref{n-tree}.
\end{proof}

Before ending this section, we give a brief calculation of $\Cfconst$ in \eqref{eq:1.3} for the simplest case:
\begin{proposition}\label{prop:srwC}
If $\brwCh$ is the geometric distribution with parameter $\frac{1}{2}$, i.e. $\brwCh(k)=2^{-k-1}$, and $\brwD$ and $\rwD$ are {one-step distributions of independent} simple random walks in $\mathbb R^6$, then $\Cfconst=9\pi^{-3}$.
\end{proposition}
\begin{proof}
Recall from \Cref{n-tree} that 
\[
\Cfconst=\sum_{k=1}^\infty (k-1)k\brwCh(k)\cdot \gnE \bracks*{\int_1^e F_{\rwD,\brwD}(B_t^\brwD)dt}.
\]
The first term is just the variance of the geometric distribution,
\[\sum_{k=1}^\infty (k-1)k\brwCh(k)=2.\]
For the second term, we first determine $F_{\rwD,\brwD}$.
Denote by $(S_n),(\tilde S_n)$ two independent simple random walks in $\mathbb R^6$ started from $0$, then by \Cref{moments_Green}, for $|z|\rightarrow \infty$,
\begin{align*}
F_{\rwD,\brwD}(z)
&=\brwE\bracks*{\sum_{\gnNd\in\gnT}\green{\rwD}(z+\gnX{\gnNd})}+O(|z|^{-3})\\
&=\mathbb E\bracks*{\sum_{n=0}^\infty\green{\rwD}(z+S_{n})}+O(|z|^{-3})\\
&=\sum_{n=0}^\infty\sum_{m=0}^\infty{\mathbb P(\tilde S_{m}=z+S_{n})}+O(|z|^{-3})\\
&=\sum_{k=0}^\infty{(k+1)\mathbb P(S_{k}=z)}+O(|z|^{-3}).
\end{align*}
Then simplify the sum by \Cref{LCLT}, we have
\[F_{\rwD,\brwD}(z)=9\pi^{-3}|z|^{-2}+O(|z|^{-3}).\]
By definition, $F_{\rwD,\brwD}(\lambda z)=\lambda^{-2}F_{\rwD,\brwD}(z)$ for any $z\ne0$, therefore
\[F_{\rwD,\brwD}(z)=9\pi^{-3}|z|^{-2},\,z\ne 0.\]
We can then conclude by the fact that for a $6$-dimensional Brownian motion with covariance matrix $\frac{1}{6}\mathtt{I}_6$,
\[
\gnE\bracks*{\int_1^e |B_t^\brwD|^{-2}dt}=\frac{1}{2}.
\]
\end{proof}
\subsection{Limit theorem for the infinite model}
In this section, we apply the estimates of Green's functions to deduce the estimates for the capacity using \Cref{cap_relation}.
We begin by estimating the term $G_n$ in \Cref{cap_relation}.

\begin{lemma}\label{mGreen}
In dimension $d=6$, let $\rwD$ be a distribution with conditions in \eqref{assumption}, and let $\greenkill{\rwD}{1-\frac{1}{n}}(x)=\sum_{i\ge0}(1-\frac{1}{n})^i\rwP{\rwD}{0}(\rwS{i}=x)$ as in \Cref{cap_relation}. 
There exists $C(\rwD)>0$ such that for all $x\in\mathbb{Z}^6$ and $n\ge 1$,
\[\green{\rwD}(x)-\greenkill{\rwD}{1-\frac{1}{n}}(x)\le \frac{C(\rwD)}{n}.\]
\end{lemma}
\begin{proof}
Since $(1-\frac{k}{n})\vee0\le (1-\frac{1}{n})^k$, 
we have  that
\[
	\greenkill{\rwD}{1-\frac{1}{n}}(x)\ge\sum_{k=0}^{n}(1-\frac{k}{n})\rwP{\rwD}{0}(\rwS{k}=x)\ge \green{\rwD}(x)-\sum_{k\in\mathbb{N}}\frac{k\wedge n}{n}\rwP{\rwD}{0}(\rwS{k}=x).
\]
Then the desired estimate follows because there exists $C(\rwD)>0$ such that $\rwP{\rwD}{0}(\rwS{k}=x)\le C(\rwD)k^{-3}$ uniformly in $x\in\mathbb Z^6$ by \Cref{LCLT}.
\end{proof}

\begin{lemma}\label{mConcentration}
In the same setting as \Cref{n-tree}, assume that $\xi_n^l,\xi_n^r$ are independent geometric random variables with parameter $\frac{1}{n}$. Set
\[
G_n:=\sum_{i=-\xi_n^l}^{\xi_n^r}\green{\rwD}^{(1-\frac 1 n)}(\hdX{i}).
\]
If $\brwCh$ has finite $5$-th moment, then as $n\rightarrow\infty$,
\[
\hdP\left(\absolute*{G_n-\frac{1}{2}\Cfconst\log n}>\epsilon\log n\right)=o((\log n)^{-2}).
\]
If $\brwCh$ has finite $(m+1)$-th moment for $m\ge 2$, then as $n\rightarrow\infty$,
\[
\hdE[(G_n)^m]=O((\log n)^m).
\]
\end{lemma}
\begin{proof}
	If $\xi_n$ is a geometric random variable with parameter $\frac{1}{n}$, it is not hard to see that
	\[
	\gnP(n(\log n)^{-3}\le \xi_n<n\log n)=1-o((\log n)^{-2}).
	\]
	Therefore,
	\begin{align*}
	&\hdP\pars*{G_n>\frac{1}{2}\Cfconst\log n+\epsilon\log n}\\
	=&\hdP\pars*{G_n>\frac{1}{2}\Cfconst\log n+\epsilon\log n,\;\xi_n^l,\xi_n^r<n\log n}+o((\log n)^{-2})\\
	\le&\hdP\left(\sum_{i=-n\log n}^{n\log n}\green{\rwD}(\hdX{i})>\frac{1}{2}\Cfconst\log n+\epsilon\log n\right)+o((\log n)^{-2})=o((\log n)^{-2}),
	\end{align*}
	where the last line follows from \Cref{core}.
	For the other side, we have that
	\begin{align*}
	&\hdP\pars*{G_n<\frac{1}{2}\Cfconst\log n-\epsilon\log n}\\
	=&\hdP\pars*{G_n<\frac{1}{2}\Cfconst\log n-\epsilon\log n,\;\xi_n^l,\xi_n^r\geq n(\log n)^{-3}}+o((\log n)^{-2})\\
	\le&\hdP\left(\sum_{i=-n(\log n)^{-3}}^{n(\log n)^{-3}}\green{\rwD}(\hdX{i})<\frac{1}{2}\Cfconst\log n-\epsilon\log n+2C(\rwD)(\log n)^{-3}\right)+o((\log n)^{-2})\\
	=&o((\log n)^{-2}),
	\end{align*}
	where $C(\rwD)$ is the constant in \Cref{mGreen}.
	
	Moreover, by \Cref{n-tree}, the $m$-th moment is bounded by
	\begin{align*}
	\hdE[(G_n)^m]
	&\le \hdE\bracks*{\left(\sum_{i=-\xi_n^l}^{\xi_n^r}\green{\rwD}(\hdX{i})\right)^m}\\
	&\le C_1(\brwCh,\brwD,\rwD)\sum_{k\ge 0}\gnP(\max(\xi_n^l,\xi_n^r)=k)(\log k)^m\le C_2(\brwCh,\brwD,\rwD)(\log n)^m.
	\end{align*}
\end{proof}

Apply \Cref{cap_relation} to the infinite model $(\hdX{i})$, now we are able to go from Green's functions estimates to {the one-point contribution in} the capacity of the infinite model. Recall that by \Cref{cap_relation}, if we set $\xi_n$ to be another independent geometric random variable with parameter $\frac{1}{n}$ and
\begin{align*}
&I_n:=\mathbf 1_{\{\hdX{i}\ne 0,0<i\le \xi^r_n\}}=\mathbf 1_{\{0\notin R[1,\xi^r_n]\}},\\
&E_n:=\rwP{\rwD}{0}\pars*{\tau^+_{\braces{\hdX{-\xi^l_n},\dots,\hdX{\xi^r_n}}}>\xi_n}=\rwP{\rwD}{0}\pars*{\tau^+_{R[-\xi^l_n,\xi^r_n]}>\xi_n},
\end{align*}

then 
\begin{equation}\label{eq:1.8}
\gnE [E_nG_nI_n]=1.
\end{equation}

\begin{lemma}\label{main}
In dimension $d=6$, let $\brwCh,\brwD,\rwD$ be distributions with assumptions in \eqref{assumption} and that $\brwCh$ has finite $5$-th moment. Recall the infinite model in \Cref{hd_model},
\begin{equation}\label{eq:1.9}
\lim_{n\rightarrow\infty}(\log n)\rwP{\rwD}{0}\otimes\hdP\left(0\not\in \hdR{1}{n},\tau^+_{\hdR{-n}{n}}=\infty\right)=2\Cfconst^{-1},
\end{equation}
where $\Cfconst$ is the constant defined by \eqref{eq:1.3}.
\end{lemma}
\begin{proof}
	For any fixed $\epsilon>0$ sufficiently small, let 
	\[
	A_{n,\epsilon}=\braces*{\absolute*{G_n-\frac{1}{2}\Cfconst\log n}\le\epsilon\log n},
	\]
	which, by \Cref{mConcentration}, happens with probability $1-o((\log n)^{-2}).$
	
	By Cauchy-Schwarz, we have that
	\[\hdE[E_nI_nG_n\mathbf 1_{A^c_{n,\epsilon}}]\le\sqrt{\hdP(A^c_{n,\epsilon})\hdE(G_n^2)}=o(1),\] 
	because $0\le E_n,I_n\le 1$ (by definition), $\hdP(A^c_{n,\epsilon})=o((\log n)^{-2})$, and $\hdE(G_n^2)=O((\log n)^2)$ by \Cref{mConcentration}. 
	This together with \eqref{eq:1.8} implies that
	\[\hdE[E_nI_nG_n\mathbf 1_{A_{n,\epsilon}}]=1-o(1).\]
	
	Moreover, since
	$0\le E_n,I_n\le1$, we have that
	\[
	\pars*{\frac{1}{2}\Cfconst-\epsilon}(\log{n})\pars*{\hdE[E_nI_n]-\hdP(A_{n,\epsilon}^c)}\le \hdE[E_nI_nG_n\mathbf 1_{A_{n,\epsilon}}]\le \pars*{\frac{1}{2}\Cfconst+\epsilon}(\log{n})\hdE[E_nI_n],
	\]	
	thus 
	\begin{align*}
	\limsup_{n\rightarrow\infty}\pars*{\frac{1}{2}\Cfconst-\epsilon}(\log{n})\hdE[E_nI_n]\le1 \\
	\liminf_{n\rightarrow\infty}\pars*{\frac{1}{2}\Cfconst+\epsilon}(\log{n})\hdE[E_nI_n]\ge1. 
	\end{align*}
	Since this holds for any $\epsilon$, we have $\frac{1}{2}\Cfconst(\log n)\hdE[E_nI_n]=1+o(1)$. 
	That is to say
	\[
	\rwP{\rwD}{0}\otimes\hdP\left(0\not\in \hdR{1}{\xi_n^r},\tau^+_{\hdR{-\xi^l_n}{\xi^r_n}}>\xi_n\right)=\frac{2+o(1)}{\Cfconst\log n}.
	\]
	Moreover, apply the simple estimate
	\[
	\gnP\pars*{n(\log n)^{-3}\le \xi_n,\xi_n^l,\xi_n^r<n\log n}=1-o((\log n)^{-2})
	\]
	for all three random variables $\xi_n,\xi_n^l,\xi_n^r$,
	by monotonicity we have that
	\[
	\rwP{\rwD}{0}\otimes\hdP\left({0\not\in \hdR{1}{n}},\,\tau^+_{\hdR{-n}{n}}\ge n\right)=\frac{2+o(1)}{\Cfconst\log n}.\]
	
	Now \eqref{eq:1.9} follows, since
	\begin{align*}
	\rwP{\rwD}{0}\otimes\hdP\left(n<\tau^+_{\hdR{-n}{n}}<\infty\right)
	&\le\sum_{k>n}\rwP{\rwD}{0}\otimes\hdP(\rwS{k}\in \hdR{-n}{n})\\
	&\lesssim\sum_{k>n}n\sup_{z\in\mathbb{Z}^6}\rwP{\rwD}{0}(\rwS{k}=z)
	\asymp n^{-1}
	\end{align*}
	is negligible, where in the last line we use \Cref{LCLT}.
	\end{proof}

Finally, we conclude the study for the capacity of the infinite model by a second moment method, analogue to \cite[Theorem 14]{LeGall-Lin-range}.

\begin{proposition}\label{thm_dim6}
In dimension $d=6$, let $\brwCh,\brwD,\rwD$ be distributions with assumptions \eqref{assumption} and that $\brwCh$ has finite $5$-th moment. Recall the infinite model in \Cref{hd_model}. As $n\rightarrow\infty$, under $\hdP$,
\[
\frac{\log n}{n}\capa\hdR{0}{n}\overset{\mathbb L^2}{\longrightarrow}2\Cfconst^{-1},
\]
where $\Cfconst$ is defined in \eqref{eq:1.3}.
\end{proposition}
\begin{proof}
Decompose the capacity as discussed in \eqref{decompose_capa}. By \eqref{invariant_hd} and \Cref{main} we have that
\begin{align*}
&\frac{\log n}{n}\hdE[\capa\hdR{0}{n}]\\
=&\frac{\log n}{n}\sum_{i=0}^n
\hdE\bracks*{
\mathbf 1_{\hdX{i}\not\in \hdR{i+1}{n}}\rwP{\rwD}{\hdX{i}}
\pars*{\tau^+_{\hdR{0}{n}}=\infty}
}
\\
=&\frac{\log n}{n}\sum_{i=0}^n
\hdE\bracks*{
\mathbf 1_{0\not\in \hdR{1}{n-i}}\rwP{\rwD}{0}
\pars*{\tau^+_{\hdR{-i}{n-i}}=\infty}
}\\
\ge&{(\log n)}{\rwP{\rwD}{0}\otimes\hdP}\pars*{0\not\in \hdR{1}{n}, \tau^+_{\hdR{-n}{n}}=\infty}
\stackrel{n\rightarrow\infty}{\longrightarrow} 2\Cfconst^{-1}.
\end{align*}

Then it suffices to show that
\begin{equation}\label{SecondMomentInfiniteModel}
\limsup_{n\rightarrow\infty}\pars*{\frac{\log n}{n}}^2\hdE\bracks*{(\capa\hdR{0}{n})^2}
\le \pars*{2\Cfconst^{-1}}^2.
\end{equation}

In fact, for any $\alpha\in(0,\frac{1}{4})$, set 
\[
D(\alpha)=\setof{(i,j)}{0<i<j<n \text{ and } i,j-i,n-j>n^{1-\alpha}},\] 
then
\begin{align*}
&\hdE[(\capa\hdR{0}{n})^2]\\
=&
\sum_{i,j=0}^n 
\hdE\bracks*{
\mathbf 1_{\hdX{i}\notin\hdR{i+1}{n}}\mathbf 1_{\hdX{j}\notin\hdR{j+1}{n}}
\rwP{\rwD}{\hdX{i}}\pars*{\tau^+_{\hdR{0}{n}}=\infty}\rwP{\rwD}{\hdX{j}}\pars*{{\tau}^+_{\hdR{0}{n}}=\infty}
}\\
=&
2\sum_{D(\alpha)} 
\hdE\bracks*{
\mathbf 1_{\hdX{i}\notin\hdR{i+1}{n}}\mathbf 1_{\hdX{j}\notin\hdR{j+1}{n}}
\rwP{\rwD}{\hdX{i}}\pars*{\tau^+_{\hdR{0}{n}}=\infty}\rwP{\rwD}{\hdX{j}}\pars*{{\tau}^+_{\hdR{0}{n}}=\infty}
}+o\pars*{\frac{n^2}{(\log n)^2}}.
\end{align*}

Moreover, write $k=j-i$ for simplicity, then for $(i,j)\in D(\alpha),$ by \eqref{invariant_hd},
\begin{align*}
&\hdE\bracks*{
\mathbf 1_{\hdX{i}\notin\hdR{i+1}{n}}\mathbf 1_{\hdX{j}\notin\hdR{j+1}{n}}
\rwP{\rwD}{\hdX{i}}\pars*{\tau^+_{\hdR{0}{n}}=\infty}\rwP{\rwD}{\hdX{j}}\pars*{{\tau}^+_{\hdR{0}{n}}=\infty}}\\
\le &\hdE\left[
\mathbf 1_{0\notin\hdR{1}{n^{1-3\alpha}}}\mathbf 1_{\hdX{k}\notin\hdR{k+1}{k+n^{1-3\alpha}}}
\times\right.\\
&\qquad\qquad\left.
\rwP{\rwD}{0}\pars*{\tau^+_{\hdR{-n^{1-3\alpha}}{n^{1-3\alpha}}}=\infty}\rwP{\rwD}{\hdX{k}}\pars*{{\tau}^+_{\hdR{k-n^{1-3\alpha}}{k+n^{1-3\alpha}}}=\infty}
\right]
\end{align*}
By \eqref{size_tree}, with probability $1-o((\log n)^{-2})$, one has $\absolute*{\zeta_{\pm n^{\frac{1}{2}-\alpha}}}\in[2n^{1-3\alpha},n^{1-\alpha}]$. And under this condition, the range $\hdR{-n^{1-3\alpha}}{n^{1-3\alpha}}$ and $\hdR{k-n^{1-3\alpha}}{k+n^{1-3\alpha}}$ correspond to disjoint subtrees in $\hdT$, thus by strong Markov property applied at the node $(n^{\frac{1}{2}-\alpha},\varnothing)$, we can bound the probability above by 
\begin{align*}
&\pars*{\rwP{\rwD}{0}\otimes\hdP(0\notin\hdR{1}{n^{1-3\alpha}},\tau^+_{\hdR{-n^{1-3\alpha}}{n^{1-3\alpha}}}=\infty)}^2+o((\log n)^{-2})\\
=&\left(\left({2\Cfconst^{-1}(1-3\alpha)^{-1}}\right)^2+o(1)\right)(\log n)^{-2}
\end{align*}
using \Cref{main}. Then \eqref{SecondMomentInfiniteModel} follows by summing over all indices in $D(\alpha)$ and let $\alpha\rightarrow 0+$. 
\end{proof}

\subsection{Proof of \texorpdfstring{\Cref{mainresult}}{} (2)}
We use the same treatment as for high dimensions to {extend} the result on the infinite model to that of a standard branching process.
\begin{theorem}\label{mainpart2}
In dimension $d=6$, assume that $\brwCh,\brwD,\rwD$ are distributions satisfying \eqref{assumption} and $\brwCh$ has finite $5$-th moment. Under the law $\brwP(\cdot|\#T=n+1)$ of the Galton-Watson tree conditioned to have $n+1$ nodes, let $\gnR[0,n]$ be the range of the branching random walk indexed by the conditioned tree, then
\[
\frac{\log n}{n}\capa(\hdR{0}{n})\rightarrow 2\Cfconst^{-1}\text{ in probability,}
\]
where $\Cfconst$ is the constant in \eqref{eq:1.3}.
\end{theorem}
\begin{proof}
As in the proof of \Cref{HighDimConclusion}, we can prove by \Cref{new_hd_2} that for any $a\in(0,1),\epsilon>0$,
\[
\lim_{n\rightarrow\infty}\brwP\parsof*{\absolute*{\frac{\log n}{n}\capa(\hdR{0}{an})-2a\Cfconst^{-1}}>\epsilon}{\#T=n+1}=0.
\]
Take $a\rightarrow 1-$, then we have a lower bound for $\capa\hdR{0}{n}$,
\[
\lim_{n\rightarrow\infty}\brwP\parsof*{\frac{\log n}{n}\capa(\hdR{0}{n})-2\Cfconst^{-1}<-\epsilon}{\#T=n+1}=0.
\]

If we reverse the order for nodes on a tree $\gnT$, and set the range of its last $an$ nodes by $\hdR{0}{an}^-$, then $\hdR{0}{an}^-$ will satisfy the same estimate as $\hdR{0}{an}$.
Moreover, $\hdR{0}{n/2},\hdR{0}{n/2}^-$ will cover all the tree expect for a negligible number of nodes (\cite[p. 20]{zhu-cbrw}), thus 
\begin{align*}
&\lim_{n\rightarrow\infty}\brwP\parsof*{\frac{\log n}{n}\capa(\hdR{0}{n})-2\Cfconst^{-1}>\epsilon}{\#T=n+1}\\
=&\lim_{n\rightarrow\infty}\brwP\parsof*{\frac{\log n}{n}\capa(\hdR{0}{n/2}\cup\hdR{0}{n/2}^-)-2\Cfconst^{-1}>\epsilon}{\#T=n+1}\\
\le&\lim_{n\rightarrow\infty}\brwP\parsof*{\frac{\log n}{n}(\capa\hdR{0}{n/2}+\capa\hdR{0}{n/2}^-)-2\Cfconst^{-1}>\epsilon}{\#T=n+1}=0.
\end{align*}
\end{proof}

\section{Open problems}\label{lowDim}
\subsection{Scaling limit in low dimensions}
In dimension $d\in\{3,4,5\}$, by the same method of \Cref{moments_Green}, we have that
\begin{enumerate}
\item 
As $z\rightarrow\infty$, we have that
\[
\brwE\bracks*{\sum_{\gnNd\in\gnT}\green{\rwD}(z+\gnX{\gnNd})}\asymp|z|^{4-d}.
\]
Therefore, for the infinite model,
\[
\hdE\bracks*{\sum_{i=-\zeta_n}^{\zeta_n}\green{\rwD}(\hdX{i})}\asymp n^{\frac{6-d}{2}}.
\]
\item For any $m\ge 2$ and $z\in\mathbb Z^d$ , 
\[\brwE\bracks*{\pars*{\sum_{\gnNd\in\gnT}\green{\rwD}(z+\gnX{\gnNd})}^m}=\infty.\]
\end{enumerate}
Divergence of variance shows that, viewing $n$ points as $\sqrt n$ subtrees is no longer a good choice. However, one can still directly estimate the sum of Green's functions for $\hdR{0}{n}$ by studying its corresponding height process $(\text{dist}(\hdX{0},\hdX{i}))$ (cf. eg. \cite[Theorem 2.1.1]{heightbook}), and show that the sum of Green's functions of $n$ points still behave like that of $\sqrt n$ subtrees in terms of expectation,
\begin{equation}\label{LowDimG}
\hdE\bracks*{\sum_{i=0}^{n}\green{\rwD}(\hdX{i})}\asymp n^{\frac{6-d}{4}}.
\end{equation}
By this estimate and \Cref{cap_relation}, we conjecture that $\capa\hdR{0}{n}$ is of the order $n^{\frac{d-2}{4}}$. Moreover, as is the case for the range of branching random walks (\cite{LeGall-Lin-lowdim}), we conjecture that
\[
n^{\frac{2-d}{4}}\capa\hdR{0}{n}\text{ converges in distribution},\,d\in\{3,4,5\}.
\]
In fact, the scaling limit for critical branching random walks for $d\le 3$ has been studied in \cite[Theorem 4]{LeGall-Lin-lowdim} in terms of local times. Yet the capacity (for bounded sets in $\mathbb R^d$) cannot be easily expressed as a function of local times, since a sphere with local time $0$ almost everywhere has the same capacity as a solid ball.

In fact, \Cref{capA} allows us to establish a lower bound in dimension $d=5$ in favor of this conjecture.
By \Cref{hd_invariant_shift} and \eqref{LowDimG}, we have that
\[
\hdE\bracks*{\sum_{x,y\in\hdR{0}{n}}\green{\rwD}(x,y)}\le n\hdE\bracks*{ \sum_{i=-n}^n\green{\rwD}(\hdX{i})}\asymp n^{\frac{5}{4}}.
\]
Moreover, by \cite[Theorem 1]{LeGall-Lin-range} we have $\hdE[\#\hdR{0}{n}]\asymp n$.
Therefore, take $A=\hdR{0}{n}$ in \Cref{capA} and take $k=C_1(\rwD,\brwCh,\brwD)n^{\frac{1}{4}}$ for some large enough constant, then
\[\hdE\bracks*{\capa \hdR{0}{n}}
\ge
\hdE\bracks*{\frac{\#\hdR{0}{n}}{k+1}}-\hdE\bracks*{\frac{\sum_{x,y\in A}\green{\rwD}(x,y)}{k(k+1)}}\ge C_2(\rwD,\brwCh,\brwD) n^{\frac{3}{4}}.\]

\subsection{Central limit theorem in high dimensions}
On the other hand, we expect a central limit theorem 
\[
\frac{\capa\hdR{0}{n}-\hdE[\capa\hdR{0}{n}]}{\sqrt n}\rightarrow\mathcal N(0,\sigma_d^2),\,{n\rightarrow\infty},
\]
\[
\frac{\capa\hdR{0}{n}-\brwE[\capa\hdR{0}{n}|\#T=n]}{\sqrt n}\rightarrow\mathcal N(0,\sigma_d^2),\,{n\rightarrow\infty},
\]
for dimension $d$ sufficiently large, as is the case for high dimensional simple random walks, cf. \cite[Theorem 1.1]{Asselah-Schapira-Sousi-capRW}.
Intuitively, in high dimensions, the capacity can be seen as the partial sum of a stationary sequence,
\begin{align*}
\capa \hdR{0}{n}
&=\sum_{i=0}^n\mathbf 1_{\braces{\hdX{i}\not\in\braces{\hdX{i+1},\dots,\hdX{n}}}}\rwP{\rwD}{\hdX{i}}(\tau_{\hdR{0}{n}}^+=\infty)\\
&\approx
\sum_{i=0}^n\mathbf 1_{\braces{\hdX{i}\not\in\braces{\hdX{i+1},\dots,\hdX{i+c_n}}}}\rwP{\rwD}{\hdX{i}}(\tau_{\hdR{i-c_n}{i+c_n}}^+=\infty),
\end{align*}
where $(c_n)$ is some increasing sequence, say $\frac{n}{\log n}$, determined by sharper estimates in high dimensions. One can then attempt to deduce a central limit theorem using properties of stationary processes, eg. \cite[Theorem 13, Corollary 29]{stationaryReview}.

{\bf Acknowledgement}. The first author would like to thank Yueyun Hu for suggesting this topic and providing valuable advice. We would also like to thank the anonymous referee for pointing out the possibility of generalization in the critical dimension.
\bibliographystyle{plain}
\bibliography{1}
\end{document}